\documentclass[reqno]{amsart}
\usepackage{amssymb,amsmath,hyperref}
\usepackage{amsrefs, float}
\usepackage[foot]{amsaddr}
\usepackage{bbold,stackrel, multicol}
\usepackage{mathrsfs}

\newcommand{\Z}{{\textsf{\textup{Z}}}}

\newtheorem{thm}{Theorem}

\newtheorem{cor}[thm]{Corollary}
\newtheorem{defi}[thm]{Definition}
\newtheorem{rem}[thm]{Remark}
\newtheorem{nota}[thm]{Notation}

\newtheorem{princ}[thm]{Principle}

\newtheorem{ack}[thm]{Acknowledgement}

\newcommand\be{\begin{equation}}
\newcommand\ee{\end{equation}} 

\usepackage{amsmath,amsfonts} 

\usepackage[applemac]{inputenc}

\def\bdefi{\begin{defi}}
\def\edefi{\end{defi}}
\def\bnota{\begin{nota}\rm}
\def\enota{\end{nota}}
\def\FIVE{\Pi_{1}^{1}\text{-\textup{\textsf{CA}}}_{0}}
\def\SIX{\Pi_{2}^{1}\text{-\textsf{\textup{CA}}}_{0}}
\def\SIXk{\Pi_{k}^{1}\text{-\textsf{\textup{CA}}}_{0}}
\def\SIXko{\Pi_{k+1}^{1}\text{-\textsf{\textup{CA}}}_{0}}
\def\SIXK{\Pi_{k}^{1}\text{-\textsf{\textup{CA}}}_{0}^{\omega}}

\def\ATR{\textup{\textsf{ATR}}}

\def\ZFC{\textup{\textsf{ZFC}}}
\def\ZF{\textup{\textsf{ZF}}}

\def\L{\textsf{\textup{L}}}

\def\RCA{\textup{\textsf{RCA}}}
\def\({\textup{(}}
\def\){\textup{)}}

\def\RCAo{\textup{\textsf{RCA}}_{0}^{\omega}}
\def\ACAo{\textup{\textsf{ACA}}_{0}^{\omega}}

\def\WKL{\textup{\textsf{WKL}}}

\def\bye{\end{document}}
\def\N{{\mathbb  N}}
\def\Q{{\mathbb  Q}}
\def\R{{\mathbb  R}}
\def\NN{{\mathfrak  N}}
\def\NNN{{\mathcal  N}}

\def\SS{\textup{\textsf{S}}}

\def\di{\rightarrow}

\def\asa{\leftrightarrow}

\def\ACA{\textup{\textsf{ACA}}}

\def\QFAC{\textup{\textsf{QF-AC}}}

\def\SIND{\Sigma\textup{\textsf{-IND}}}

\def\open{\textup{\textsf{open}}}

\def\BOOT{\textup{\textsf{BOOT}}}

\def\IND{\textup{\textsf{IND}}}

\def\RANGE{\textup{\textsf{RANGE}}}

\def\fin{\textup{\textsf{fin}}}

\def\CAC{\textup{\textsf{CAC}}}

\def\CAC{\textup{\textsf{CAC}}}

\def\eps{\varepsilon}

\def\ADS{\textup{\textsf{ADS}}}

\usepackage{graphicx}
\usepackage{tikz}
\usetikzlibrary{matrix}
\usepackage{comment,tikz-cd}

\numberwithin{equation}{section}
\numberwithin{thm}{section}

\begin{document}
\title{Coding is hard}
\author{Sam Sanders}
\address{Department of Philosophy II, RUB Bochum, Germany}
\email{sasander@me.com}
\begin{abstract}
A central topic in mathematical logic is the classification of theorems from mathematics in hierarchies according to their logical strength.  
Ideally, the place of a theorem in a hierarchy does not depend on the representation (aka coding) used.  
In this paper, we show that the standard representation of compact metric spaces in second-order arithmetic has a profound effect.  
To this end, we study basic theorems for such spaces like \emph{a continuous function has a supremum} and \emph{a countable set has measure zero}.  
We show that these and similar third-order statements imply at least Feferman's highly non-constructive \emph{projection principle}, and even full second-order arithmetic or countable choice in some cases.
When formulated with representations (aka codes), the associated second-order theorems are provable in rather weak fragments of second-order arithmetic.  
Thus, we arrive at the slogan that 
\begin{center}
\emph{coding compact metric spaces in the language of second-order arithmetic can be as hard as second-order arithmetic or countable choice}.
\end{center}
We believe every mathematician should be aware of this slogan, as central foundational topics in mathematics make use of the standard second-order representation of compact metric spaces. 
In the process of collecting evidence for the above slogan, we establish a number of equivalences involving Feferman's projection principle and countable choice.
We also study generalisations to fourth-order arithmetic and beyond with similar-but-stronger results.  
\end{abstract}

\setcounter{page}{0}
\tableofcontents
\thispagestyle{empty}
\newpage

\maketitle
\thispagestyle{empty}

\section{Introduction and preliminaries}
\subsection{Aim}
The study of compact metric spaces in mathematical logic generally proceeds via `representations' or `codes' in the language of second-order arithmetic, where the latter only has variables for natural numbers and sets of natural numbers.   
In this paper, we collect evidence for the following slogan: 
\begin{center}
\emph{coding compact metric spaces in the language of second-order arithmetic can be as hard as full second-order arithmetic or countable choice}.
\end{center}
In particular, we identify a number of well-known theorems about compact metric spaces that imply or are equivalent to strong axioms when formulated without codes/representations, sometimes even full second-order arithmetic and countable choice\footnote{The Axiom of Choice restricted to countable collections is not provable in $\ZF$, i.e.\ the usual foundations of mathematics $\ZFC$ without the Axiom of Choice (see e.g.\ \cite{heerlijkheid}).}.  
These basic statements include the supremum principle for (Lipschitz) continuous functions, that countable sets have measure zero, the intermediate value theorem, and even the existence of representations themselves via separability.
By contrast, when formulated with codes/representations, the associated second-order theorems can be proved in rather weak fragments of second-order arithmetic. 
Thus, our results provide an answer to to the following question by Friedman.
\begin{quote}
Are there any ordinary mathematical theorems involving third order
objects without presentations that have any inherent logical strength? (original in all-caps, FOM mailing list, \cite{friedfom2})
\end{quote}
We believe that every mathematician should be aware of the above centred slogan, as major foundational topics in mathematics, like \emph{Turing computability} (\cites{zweer, cooper1, recmath1, recmath2}), \emph{computable analysis} (\cites{brat, wierook, brapress, bratger}) and \emph{Reverse Mathematics} (\cites{simpson2, stillebron, damurm}) make essential use of second-order codes/representations.  The associated classification apparently depends greatly on this coding practise and `less coding-heavy' alternatives should be considered, in our humble opinion.   

\smallskip

A less foundational result in this paper is the observation that basic statements about compact metric spaces, like that a (Lipschitz) continuous function is bounded or the existence of discontinuous functions, are equivalent to a known\footnote{The fragment at hand is called $\QFAC^{0,1}$ and not provable in $\ZF$, i.e.\ the usual foundations of mathematics without the Axiom of Choice (see \cite{kohlenbach2} and Section \ref{lll}).} fragment of the Axiom of countable Choice, namely by Theorems \ref{NNZ} and \ref{zirf}.   

\smallskip

An important point regarding the above centred slogan is that the results in this paper are \emph{robust}, i.e.\ we still obtain second-order arithmetic or countable choice for many variations of the theorems at hand.  Robustness is studied as a property of logical systems in mathematical logic as follows in \cite{montahue}*{p.\ 432}.  
\begin{quote}
[\dots] one distinction that I think is worth making is the one between robust systems and non-robust systems. A system is robust if it is equivalent to small perturbations of itself. This is not a precise notion yet, but we can still recognize some robust systems. 
\end{quote}
Finally, the main aim of this paper is to identify basic third-order statements about compact metric spaces that imply strong axioms.  
Nonetheless, the associated proofs provide templates that yield similar results for fourth- and higher-order statements.  For completeness, we study
two examples in Section \ref{zonggg}, among which the intermediate value theorem on connected metric spaces.  

\smallskip

We provide a more detailed introduction in Section \ref{fintro} while Section \ref{pprel} lists required definitions and axioms.  
Our main results are in Sections \ref{main} and \ref{zonggg}.  
\subsection{Introduction}\label{fintro}
While set theory provides a foundation for most of mathematics, so-called ordinary mathematics is often studied using the more frugal system \emph{second-order arithmetic} $\Z_{2}$ (\cite{simpson2}*{I-II}).  
A central topic is to identify the minimal axioms needed to prove a given theorem of ordinary mathematics, i.e.\ which subsystem of $\Z_{2}$ suffices for a proof.  
This is the aim of \emph{Reverse Mathematics} (\cite{simpson2, stillebron, damurm}) where a major result is that the majority of theorems of ordinary mathematics are provable in rather weak fragments of $\Z_{2}$ carrying foundational import.   

\smallskip

Now, the state-of-the-art described in the previous paragraph is an instance of a central topic of mathematical logic, namely the classification of theorems from mathematics in hierarchies according to their logical strength.  
Ideally, the place of a theorem in these hierarchies does not depend too much on the representation used.  Since the language of $\Z_{2}$, called $\L_{2}$, only has variables for natural numbers $n\in \N$ and sets $X\subset \N$, higher-order objects have to be `represented' or `coded' via second-order objects.  Prominent third-order examples are continuous functions on the reals, open and closed sets of reals, and compact metric spaces.   

\smallskip

The goal of this paper is to identify basic third-order theorems that imply strong axioms, including full second-order arithmetic and countable choice.  
In particular, we study the following elementary statements for compact metric spaces $(M, d)$ where $M$ is a subset of the Baire space.
\begin{enumerate}
\item[(a)] A compact metric space $(M, d)$ is separable.  
\item[(b)] A countable set in a compact metric space $(M, d)$ has measure zero.  
\item[(c)] A (Lipschitz) continuous function on $(M, d)$ has a supremum.
\item[(d)] An open set in $(M, d)$ equals the countable union of open balls.
\end{enumerate}
When formulated with representations/codes, the associated second-order theorems are provable in rather weak fragments of second-order arithmetic (\cite{browner, simpson2}).  
By contrast, the third-order theorems in items (a)-(d) imply at least Feferman's highly non-constructive \emph{projection principle} (see \cite{littlefef} and Section \ref{mintro}) and even full second-order arithmetic or countable choice in some cases, as established in Section~\ref{main}.  
In light of the following quote by Friedman, it seems the aforementioned results are both novel and surprising. 
\begin{quote}
I have not seen any indication of an extension of say even $\ATR_{0}$ by
fundamental mathematical statements in types 1,2,3 in the classical
mathematical literature that carries any unusual logical strengths
like higher fragments of $\Z_2$ (added LaTeX notation, \cite{friedfom1}).
\end{quote}
To establish our results, we work in Kohlenbach's base theory $\RCAo$ of higher-order Reverse Mathematics (abbreviated RM), introduced in \cite{kohlenbach2}.  We assume familiarity with $\RCAo$ and the associated RM of real analysis (\cite{kohlenbach2}*{\S2} or \cite{dagsamXIV}).

\smallskip

To be absolutely clear, $\RCAo$ is a weak logical system that we assume as a `background theory'.  In the latter, we prove that the above statements (a)-(d) imply or are equivalent to strong axioms, including even second-order arithmetic $\Z_{2}$ and countable choice $\QFAC^{0,1}$ (see Section \ref{lll} for the latter).  Along the way, we will obtain a number of elegant equivalences for Feferman's projection principle and countable choice $\QFAC^{0,1}$ in Kohlenbach's higher-order RM.  

\smallskip

Finally, items (a)-(d) deal with compact metric spaces $(M, d)$ where $M$ is a subset of the Baire space, i.e.\ these can be formulated in the language of third-order arithmetic. 
As it turns out, the associated proofs provide a kind of template that yields similar results for fourth- and higher-order arithmetic.  
We sketch such results in Section \ref{zonggg}, for completeness, noting that the third-order case is vastly more natural from the point of view of ordinary mathematics.

\smallskip

In conclusion, the coding practise involving $\L_{2}$, the language of second-order arithmetic $\Z_{2}$, can have a tremendous influence on the logical strength of basic theorems.
The results in this paper are novel since we show that this logical strength can vary as much as $\Z_{2}$ itself, or even require countable choice.  
Moreover, items (a)-(d) are rather elementary and found in many an undergraduate curriculum or textbook.  
We believe that many variations on these results are possible and look forward to the associated exploration with collaborators. 

\subsection{Preliminaries}\label{pprel}
We introduce the mostly standard definition of metric space in higher-order arithmetic (Section \ref{mez}) and some required axioms (Section \ref{lll}).
We assume basic familiarity with the formalisation of the real numbers, which is the same in second- and higher-order RM (see \cite{simpson2}*{II.5} or \cite{kohlenbach2}*{\S2}).

\subsubsection{Metric spaces in higher-order order arithmetic}\label{mez}
We discuss the definition of metric space in higher-order arithmetic, as well as some related definitions.
We note that sets are given by characteristic functions, well-known from measure and probability theory, and second-order RM, as can be gleaned from \cite{simpson2}*{X.1.12}.

\smallskip

First of all, we shall study metric spaces $(M, d)$ as in Definition~\ref{donkc}.  We assume that $M$ comes with its own equivalence relation `$=_{M}$'.  
We recall that the study of metric spaces in second-order RM is at its core based on equivalence relations, as stated explicitly in the two textbooks on RM, namely \cite{simpson2}*{I.4} or \cite{damurm}*{\S10.1}.
Now, the notation `$F:M\di \R$' denotes a function $F$ from $M$ to the reals that satisfies the following instance of the axiom of function extensionality:
\be\tag{\textup{\textsf{E}}$_{M}$}\label{koooooo}
(\forall x, y\in M)(x=_{M}y\di F(x)=_{\R}F(y)).
\ee
In particular, each component of a metric $d:M^{2}\di \R$ is assumed to satisfy \eqref{koooooo}.  
\bdefi\label{donkc}
A functional $d: M^{2}\di \R$ is a \emph{metric on $M$} if it satisfies the following properties for $x, y, z\in M$:
\begin{enumerate}
 \renewcommand{\theenumi}{\alph{enumi}}
\item $d(x, y)=_{\R}0 \asa  x=_{M}y$,
\item $0\leq_{\R} d(x, y)=_{\R}d(y, x), $
\item $d(x, y)\leq_{\R} d(x, z)+ d(z, y)$.
\end{enumerate}
\edefi 
We recall that compact metric spaces are at most the size of the set of reals by \cite{buko}*{Theorem~3.13}.
Motivated by this observation, we shall only study metric spaces $(M, d)$ with $M\subset \N^{\N}$ or $M\subset \R$, \emph{except} in Section \ref{zonggg}. 
In the latter, we show that if $M$ has a higher rank, the associated theorems are \emph{much} stronger. 

\smallskip

A subset $X\subset M$ is defined via its characteristic function $\mathbb{1}_{X}:M\di \{0, 1\}$, keeping in mind \eqref{koooooo}.  
Moreover, we use standard notation like $B_{d}^{M}(x, r)$ to denote the open ball $\{y\in M: d(x, y)<_{\R}r\}$, while $\overline{B}_{d}^{M}(x, r)$ is the associated closed ball.  
A set $O\subset M$ is open if for $x\in O$, there exists $N\in \N$ with $B_{d}^{M}(x, \frac{1}{2^{N}})\subset O$.    
A set $C\subset M$ is closed if the complement $M\setminus C$ is open. 

\smallskip

Secondly, the following definitions are mostly standard, where we note that a different nomenclature is sometimes used in the logical literature.
\bdefi[Compactness and around]\label{deaco} 
For a metric space $(M, d)$, we say that
\begin{itemize}
\item a set $X\subset M$ is \emph{finite} if there is $N\in \N$ such that for any pairwise different $x_{0}, \dots, x_{N}\in M$, there is $i\leq N$ with $x_{i}\not\in X$.  We then write $|X|\leq N$.
\item $(M, d)$ is \emph{countably-compact} if for any sequence $(O_{n})_{n\in \N}$ of open sets in $M$ such that $M\subset \cup_{n\in \N}O_{n}$, there is $m\in \N$ such that  $M\subset \cup_{n\leq m}O_{n}$,
\item $(M, d)$ is \(open-cover\) \emph{compact} in case for any $\Psi:M\di \R^{+}$, there are $x_{0}, \dots, x_{k}\in M$ such that $\cup_{i\leq k}B_{d}^{M}(x_{i}, \Psi(x_{i}))$ covers $M$, 
\item $(M, d)$ is \emph{sequentially compact} if any sequence has a convergent sub-sequence,
\item $(M, d)$ is \emph{limit point compact} if any infinite set in $M$ has a limit point,
\item a sequence $(x_{n})_{n\in \N}$ in $(M, d)$ is \emph{Cauchy} if $(\forall k\in \N)(\exists N\in \N)(\forall m, n\geq N)( d(x_{n}, x_{m})<\frac{1}{2^{k}})$,
\item $(M, d)$ is \emph{complete} in case every Cauchy sequence converges,
\item $(M, d)$ is \emph{totally bounded} if for all $k\in \N$, there are $x_{0}, \dots, x_{m}\in M$ such that $\cup_{i\leq m}B_{d}^{M}(x_{i}, \frac{1}{2^{k}})$ covers $M$,
\item $(M, d)$ is \emph{separable} if there is a sequence $(x_{n})_{n\in \N}$ in $M$ such that $(\forall x\in M, k\in \N)(\exists m\in \N  )(d(x, x_{m})<\frac{1}{2^{k}})$,
\item $(M, d)$ is \emph{bounded} if $(\exists N\in \N)(\forall x, y\in M)( d(x, y)\leq N)$,
\end{itemize}
\edefi
It is well-known that for metric spaces, the compactness notions are equivalent and that compact metric spaces are separable.  
We could study advanced notions like the Lindel\"of property, but stick with concepts studied in second-order RM. 

\subsubsection{Some axioms of higher-order arithmetic}\label{lll}
We introduce some higher-order axioms needed in the below.  We assume basic familiarity with Kohlenbach's {base theory} $\RCAo$ of higher-order RM (see \cite{kohlenbach2}*{\S2}).  

\smallskip
\noindent
First of all, the functional $E$ in $(\exists^{2})$ is also called \emph{Kleene's quantifier $\exists^{2}$}: 
\be\label{muk}\tag{$\exists^{2}$}
(\exists E:\N^{\N}\di \{0,1\} )(\forall f \in \N^{\N})\big[(\exists n\in \N)(f(n)=0) \asa E(f)=0    \big]. 
\ee
Related to $(\exists^{2})$, the functional $\mu^{2}$ in $(\mu^{2})$ is called \emph{Feferman's $\mu$} (see \cite{avi2}) and may be found -with this symbol- in Hilbert-Bernays' Grundlagen (\cite{hillebilly2}*{Supplement IV}):
\[
\mu(f):= 
\begin{cases}
n & \textup{if $n$ is the least natural number such that $f(n)=0$ }\\
0 & \textup{if $f(n)>0$ for all $n\in \N$}
\end{cases}
.
\]
We have $(\exists^{2})\asa (\mu^{2})$ over $\RCAo$ (see \cite{kohlenbach2}*{\S3}) and $\ACAo\equiv\RCAo+(\exists^{2})$ proves the same sentences as $\ACA_{0}$ by \cite{hunterphd}*{Theorem~2.5}. 

\smallskip

Secondly, the functional $\SS^{2}$ in $(\SS^{2})$ is called \emph{the Suslin functional} (\cite{kohlenbach2}):
\be\tag{$\SS^{2}$}
(\exists\SS^{2}:\N^{\N}\di \{0,1\})(\forall f \in \N^{\N})\big[  (\exists g\in \N^{\N})(\forall n \in \N)(f(\overline{g}n)=0)\asa \SS(f)=0  \big].
\ee
The system $\FIVE^{\omega}\equiv \RCAo+(\SS^{2})$ proves the same $\Pi_{3}^{1}$-sentences as $\FIVE$ by \cite{yamayamaharehare}*{Theorem 2.2}.   
By definition, the Suslin functional $\SS^{2}$ can decide whether a $\Sigma_{1}^{1}$-formula as in the left-hand side of $(\SS^{2})$ is true or false.   We similarly define the functional $\SS_{k}^{2}$ which decides the truth or falsity of $\Sigma_{k}^{1}$-formulas from $\L_{2}$; we also define 
the system $\SIXK$ as $\RCAo+(\SS_{k}^{2})$, where  $(\SS_{k}^{2})$ expresses that $\SS_{k}^{2}$ exists.  
We note that the operators $\nu_{n}$ from \cite{boekskeopendoen}*{p.\ 129} are essentially $\SS_{n}^{2}$ strengthened to return a witness (if existent) to the $\Sigma_{n}^{1}$-formula at hand.  

\smallskip

\noindent
Thirdly, full second-order arithmetic $\Z_{2}$ is readily derived from $\cup_{k}\SIXK$, or from:
\be\tag{$\exists^{3}$}
(\exists E:(\N^{\N}\di \N)\di \N)(\forall Y:\N^{\N}\di \N)\big[  (\exists f\in  \N^{\N})(Y(f)=0)\asa E(Y)=0  \big], 
\ee
and we therefore define $\Z_{2}^{\Omega}\equiv \RCAo+(\exists^{3})$ and $\Z_{2}^\omega\equiv \cup_{k}\SIXK$, which are conservative over $\Z_{2}$ by \cite{hunterphd}*{Cor.\ 2.6}. 
The functional from $(\exists^{3})$ is also called `Kleene's quantifier $\exists^{3}$', and we use the same convention for other functionals.  

\smallskip

Fourth, many results in higher-order RM are established in $\RCAo$ plus the following special case of the Axiom of (countable) Choice (\cite{kohlenbach2, dagsamV}).
\begin{princ}[$\QFAC^{0,1}$] Let $\varphi$ be quantifier-free and such that $(\forall n\in \N)(\exists f\in \N^{\N})\varphi(f, n)$.  
Then there is a sequence $  (f_{n})_{n\in \N}$ in $\N^{\N}$ with $(\forall n\in \N)\varphi(f_{n}, n)$.
\end{princ}
The local equivalence between sequential and `epsilon-delta' continuity cannot be proved in $\ZF$, but can be established in $\RCAo+\QFAC^{0,1}$ (see \cite{kohlenbach2}).
Thus, it should not be a surprise that the latter system is often used as a base theory too.
Below, we show that $\QFAC^{0,1}$ is equivalent to various basic statements about compact metric spaces, including that continuous functions are bounded on the latter.  
By contrast, the aforementioned local equivalence deals with potentially \emph{discontinuous} functions on the reals.  

\smallskip

Finally, certain equivalences in second-order RM require extra induction beyond what is available in $\RCA_{0}$ (see \cite{neeman}). 
We sometimes use extra induction in higher-order RM, like the following fragment. 
\begin{princ}[$\Sigma$-$\IND$]
The induction axiom for $\varphi(n)$ defined as $(\exists f\in \N^{\N})(Y(f, n)=0)$ for any $Y^{2}$.  
\end{princ}
One readily\footnote{For finite $X\subset M$, let $\varphi(n)$ express that there is $w^{1^{*}}$ of length $|w|=n$ consisting of distinct elements of $X$.} shows that $\ACAo+\SIND$ proves that finite sets of reals can be enumerated as a finite sequence. 

\section{Main results}\label{main}
\subsection{Introduction}\label{mintro}
In this section, we establish the results sketched in Section \ref{fintro}, i.e.\ that basic theorems concerning compact metric spaces imply or are equivalent to strong axioms, including Feferman's projection principle, $\Z_{2}$, and countable choice.  

\smallskip

First of all, Feferman introduces the `projection principle' \textsf{Proj}$_{1}$ in \cite{littlefef}*{\S5} as a third-order version of Kleene's quantifier $(\exists^{3})$ from Section \ref{lll}.  Working over a base theory akin to $\ACAo$, it is then shown that \textsf{Proj}$_{1}$
implies various well-known theorems of analysis, like the supremum principle.  Moreover, \textsf{Proj}$_{1}$ also yields $\Z_{2}$ when combined with $(\mu^{2})$.  Thus, \textsf{Proj1} can be said to be impredicative and highly non-constructive.   
Now, Feferman's language is slightly richer than that of $\ACAo$ and the following axiom constitutes the higher-order RM version of \textsf{Proj}$_{1}$:
\be\tag{$\BOOT$}
(\forall Y:\N^{\N}\di \N)(\exists X\subset \N)(\forall n\in \N  )[ n\in X\asa (\exists f\in \N^{\N})(Y(f, n)=0) ].
\ee
The name refers to the verb `to bootstrap' as $\SIXK$ plus $\BOOT$ proves $\SIXko$.  Convergence theorems for nets are equivalent to $\BOOT$, as well as the supremum principle for certain weak continuity notions (\cites{samhabil, samRMDECO}).

\smallskip

Secondly, we shall derive $\BOOT+\QFAC^{0,1}$ from the following basic third-order statements and variations for $M\subset \N^{\N}$, sometimes even obtaining equivalences.  
\begin{itemize}
\item A compact metric space is separable (Section \ref{BS}).
\item In a compact metric space, a countable set has measure zero (Section \ref{BS2}).
\item A Lipschitz continuous function on a compact metric space $(M, d)$ has a supremum (Section \ref{BSS}).
\item An open set in a compact metric space $(M, d)$ equals the countable union of open balls (Section \ref{sleepwithoneeye}).
\end{itemize}
The third item generalised to $M\subset \N^{\N}\times \N^{\N}$ implies $\Z_{2}$ by Corollary \ref{NNcozz}.  Our formulations of the supremum principle 
can be found throughout the literature, as evidenced by exact references in Section \ref{BSS}.     

\smallskip

Finally, we shall work over $\ACAo$ for convenience, but could in principle obtain most results over $\RCAo$ or $\RCAo+\WKL_{0}$ using the following trick.
\begin{rem}[On the law of excluded middle]\label{LEM}\rm
Our starting point is Kleene's arithmetical quantifier $(\exists^{2})$ from Section \ref{lll}.
By \cite{kohlenbach2}*{Prop.\ 3.12}, $(\exists^{2})$ is equivalent over $\RCAo$ to the statement 
\begin{center}
\emph{There exists an $\R\di \R$-function that is not continuous.}
\end{center}
Clearly, $\neg(\exists^{2})$ is then equivalent to \emph{Brouwer's theorem}, i.e.\ the statement that all $\R\di \R$-functions are continuous. 
Now, if we wish to prove a given statement $T$ of real analysis about possibly discontinuous functions in $\RCAo+\WKL_{0}$, we may invoke the law of excluded middle as in $(\exists^{2})\vee \neg(\exists^{2})$.  
We can then split the proof of $T$ in two cases: one assuming $\neg(\exists^{2})$ and one assuming $(\exists^{2})$.  
In the latter case, since $(\exists^{2})\di\ACA_{0}$, we have access to much more powerful tools (than just $\WKL_{0}$).  
In the former case, since $\neg(\exists^{2})$ implies that all functions are continuous, we only need to establish $T$ restricted to the special case of continuous functions.  
Moreover, we can use $\WKL_{0}$ to provide codes for all (continuous) functions (see \cite{dagsamXIV}*{\S2}).  After that, we can use the second-order RM literature to establish $T$ restricted to codes for continuous functions, and hence $T$. 
To be absolutely clear, the `law of excluded middle trick' is the above splitting of proofs based on $(\exists^{2})\vee \neg(\exists^{2})$.
\end{rem}

\subsection{Metric spaces and separability}\label{BS}
We derive Feferman's projection principle $\BOOT$ from the basic statement \emph{a compact metric space is separable}.  
The latter expresses that the metric space at hand has a second-order code/representation, following the coding of metric spaces in second-order RM (\cites{simpson2, browner}).
We also derive countable choice as in $\QFAC^{0,1}$ from the aforementioned statement and even obtain an elegant equivalence involving boundedness.  

\smallskip  

Secondly, we have the following theorem where the first item is a special case of \cite{browner}*{Theorem 3.14.ii or 3.17.ii} without codes.
As an aside, limit point compactness goes back to Weierstrass, according to Jordan (see \cite{jordel2}*{p.\ 73}).
\begin{thm}[$\ACAo+\SIND$]\label{NN}
The following statement implies $\BOOT$ and $\QFAC^{0,1}$.
\begin{itemize}
\item For any metric $d:(\N^{\N}\times \N^{\N})\di \R$ and metric space $(\N^{\N}, d)$, sequential compactness implies separability. 
\end{itemize}
We can replace `sequential' by `limit point' or `countable' or `open cover'.  
\end{thm}
\begin{proof}
The principle $\BOOT$ is equivalent to $\RANGE$ as follows:
\be\label{myhunt}\tag{$\RANGE$}
(\forall G:\N^{\N}\di \N)(\exists X\subset \N)(\forall n \in \N)\big[n\in X\asa (\exists f\in \N^{\N})(G(f)=n)  ].
\ee
Indeed, the forward direction is immediate, while for the reverse direction, define $G^{2}$ as follows for $n^{0}$ and $g^{1}$: put $G(\langle n\rangle *g)=n+1$ if $Y(g, n)=0$, and $0$ otherwise. 
Let $X\subseteq \N$ be as in $\RANGE$ and note that 
\[
(\forall m \geq 1 )( m\in X \asa (\exists f\in \N^{\N})(G(f)=m)\asa (\exists g\in \N^{\N})(Y(g, m-1)=0)  ),
\]
which is as required for $\BOOT$ after trivial modification. 

\smallskip

Now fix $G:\N^{\N}\di \N$ and define $d:(\N^{\N}\times \N^{\N})\di \R$ as follows:  $d(f, g):=|\frac{1}{2^{G(f)}}-\frac{1}{2^{G(g)}}|$ for $f, g\ne 00\dots$, $d(00\dots, f)=d(f, 00\dots):= \frac{1}{2^{G(f)}} $ for $f\ne 00\dots$, and $d(00\dots, 00\dots):=0$. 
Then $(M, d)$ is a metric space for $M=\N^{\N}$ if we define the equivalence relation `$=_{M}$' as follows:
\be\label{bookana}
f=_{M}g\equiv\big[ [f,g\ne 00\dots \wedge G(f)=G(g)] \vee f=g=00\dots\big].  
\ee
Indeed, the first two items of Definition \ref{donkc} hold by definition while for the third item, observe that for $f, g\ne 00\dots$, we have 
\[\textstyle
d(f, g)=|\frac{1}{2^{G(f)}}-\frac{1}{2^{G(g)}}|\leq |\frac{1}{2^{G(f)}}-\frac{1}{2^{G(h)}}|+ |\frac{1}{2^{G(h)}}-\frac{1}{2^{G(g)}}| = d(f, h)+d(h, g), 
\]
for $h\ne 00\dots$;  also $d(f, g)=|\frac{1}{2^{G(f)}}-\frac{1}{2^{G(g)}}|\leq \frac{1}{2^{G(f)}}+\frac{1}{2^{G(g)}} = d(f, 00\dots)+d(00\dots, g)$.  
Moreover, $d(f, 00\dots)=\frac{1}{2^{G(f)}}\leq \frac{1}{2^{G(g)}} + |\frac{1}{2^{G(f)}}-\frac{1}{2^{G(g)}}| =d(g, 00\dots)+d(f, g)$ for $g\ne 00\dots$; we also have $0=d(00\dots, 00\dots)\leq d(00\dots, f)+d(f, 00\dots)$ for any $f\in \N^{\N}$, i.e.\ Definition \ref{donkc} is satisfied.

\smallskip

To show that $(M, d)$ is sequentially compact, let $(f_{n})_{n\in \N}$ be a sequence in $M$.  In case $00\dots$ occurs infinitely many times or if $(\exists m\in \N)(\forall n\in \N)(G(f_{n})\leq m)$, there is a trivial constant sub-sequence.  
In case $(\forall m\in \N)(\exists n\in \N)(G(f_{n})> m)$, let $(g_{n})_{n\in \N}$ be a sub-sequence with $G(g_{n})>n$ for all $n\in \N$.  Clearly, $(g_{n})_{n\in \N}$ converges to $00\dots$ in $(M, d)$.

\smallskip

To show that $(M, d)$ is (countably) compact, note that any open ball containing $00\dots$ covers all but finitely elements of $M$. 
Here, $\SIND$ seems needed to enumerate the latter and obtain a finite sub-covering.
Similarly, for an infinite set $X\subset M$, $00\dots $ is a limit point of $X$.

\smallskip

Regarding the first item, let $(h_{m})_{m\in \N}$ be a sequence as provided by the separability of $(M, d)$, i.e.\ $(\forall f\in M, k\in \N)(\exists m\in \N)(d(f, h_{m})<\frac{1}{2^{k}})$.  
Then we have, for any $n\in \N$ that
\be\label{zeng}
(\exists f\in \N^{\N})(G(f)=n) \asa [ G(00\dots)=n \vee (\exists m\in \N)(G(h_{m})=n) ].
\ee
Since the right-hand side of \eqref{zeng} is arithmetical, $\RANGE$ follows from the first item.  To obtain $\QFAC^{0,1}$, we recall that $\BOOT\asa \RANGE$ as established in the first paragraph.
By the latter, it suffices to establish that for $G^{2}$, $(\forall n\in \N)(\exists f\in \N^{\N})(G(f)=n)$ implies $(\exists (f_{n})_{n\in \N})(\forall n\in \N)(G(f_{n})=n)$.
However, this readily follows from \eqref{zeng}. 
\end{proof}
Clearly, the metric space $(M, d)$ from the previous proof has a nice order structure, i.e.\ we could restrict to certain ordered metric spaces.  

\smallskip

There is nothing special about Baire space in Theorem \ref{NN}, by the following. 
\begin{cor}[$\ACAo+\SIND$]\label{NNcor}
The following item implies $\BOOT$.
\begin{itemize}
\item For any metric $d:[0,1]^{2}\di \R$ and metric space $([0,1], d)$, sequential compactness implies separability. 
\end{itemize}
We can replace `sequential' by `limit point' or `countable' or `open cover'.  
\end{cor}
\begin{proof}
First of all, we note (the well-known fact) that $\exists^{2}$ can convert a real $x\in [0,1]$ to a unique\footnote{In case of two binary representations, chose the one with a tail of zeros.} binary representation.  
Similarly, an element of Baire space $f\in \N^{\N}$ can be identified with its graph as a subset of $\N\times \N$, which can be represented as an element of Cantor space.  
Thus, $\BOOT$ is equivalent to the statement that for any $f:(\R\times \N)\di \N$, there is $X\subset \N$ such that for all $n\in \N$:
\be\label{alti2}
(\exists x\in [0,1])(f(x,n)=0)\asa n\in X.
\ee
The proof of Theorem \ref{NN} now goes through for `$\N^{\N}$' replaced by `$[0,1]$' everywhere. 
One could also use the previous observations to define a mapping from $[0,1]$ to $\N^{\N}$ and (mostly) copy the proof of Theorem \ref{NN}.  
\end{proof}
One could obtain $\QFAC^{0,1}$ in the previous corollary, as for the theorem.

\smallskip

Finally, that the centred statement in Theorem \ref{NN} implies $\QFAC^{0,1}$ is perhaps not that surprising due to the `sequential' nature of separability. 
By the following theorem, the much weaker property of boundedness suffices to obtain $\QFAC^{0,1}$. 
We seem to need $\QFAC^{0,1}_{\fin}$ as in item (b.1) which follows from the induction axiom.  Fragments of the latter are used in second-order RM in an essential way (\cite{neeman}).
\begin{thm}[$\ACAo$]\label{NNZ}
The following statements are equivalent.
\begin{enumerate}
\renewcommand{\theenumi}{\alph{enumi}}
\item The axiom of countable choice as in $\QFAC^{0,1}$.
\item The combination of the following.
\begin{itemize} 
\item[(b.1)] The axiom $\QFAC^{0,1}_{\fin}$: for quantifier-free $\varphi$ such that $(\forall n\in \N)(\exists f\in \N^{\N})\varphi(f, n)$, we have $(\forall k\in  \N)(\exists (f_{n})_{n\in \N})(\forall n\leq k)\varphi(f_{n}, n)$.
\item[(b.2)] For any metric $d:(\N^{\N}\times \N^{\N})\di \R$ and metric space $(M, d)$ with $M\subset \N^{\N}$, sequential compactness implies boundedness. 
\end{itemize}
\item The combination of the following.
\begin{itemize} 
\item[(c.1)] The axiom $\QFAC^{0,1}_{\fin}$ from item \textup{(b.1)}.
\item[(c.2)] For any metric $d:(\N^{\N}\times \N^{\N})\di \R$ and metric space $(M, d)$ with $M\subset \N^{\N}$, sequential compactness implies countable compactness. 
\end{itemize}
\end{enumerate}
\end{thm}
\begin{proof}
First of all, apply $\QFAC^{0,1}$ to `$(M, d)$ is unbounded' and note that the resulting sequence yields a sequence $(x_{n})_{n\in \N}$ such that $d(x_{0}, x_{n})>n$ for all $n\geq 1$.
Clearly, such a sequence cannot have a convergent sub-sequence, as required for item (b.2).  Of course, $\QFAC^{0,1}$ implies $\QFAC^{0,1}_{\fin}$ as in item (b.1).  
Similarly, for a countable covering $(O_{n})_{n\in \N}$, applying $\QFAC^{0,1}$ to `$(\forall n\in \N)(\exists x\in M)(x\not \in \cup_{m\leq n}O_{n})$', one obtains a sequence 
that cannot have a convergent sub-sequence, i.e.\ item (c.2) also follows from $\QFAC^{0,1}$.  

\smallskip

For the remaining implications, assume $\QFAC^{0,1}$ is false, i.e.\ there is quantifier-free $\varphi$ with $(\forall n\in \N)(\exists f\in \N^{\N})\varphi(f, n)$, but there is no sequence $(f_{n})_{n\in \N}$ with $(\forall n\in \N)\varphi(f_{n}, n)$.
Now define $f\in M$ in case $f(0)=n_{0}$ and $f=\langle n_{0}\rangle*g$ with $g=g_{0}\oplus g_{1}\oplus \dots \oplus g_{n_{0}}$ and $(\forall i\leq n_{0})(\varphi(g_{i}, i))$.  We put $f=_{M}g$ in case $f(0)=g(0)$ and define $d:(\N^{\N}\times \N^{\N})\di \R$ as $d(f, g)=|f(0)-g(0)|$.  Then $(M, d)$ is readily seen to be an unbounded and not countably compact metric space, using $\QFAC^{0,1}_{\fin}$.  To show that $(M, d)$ is sequentially compact, let $(f_{n})_{n\in \N}$ be a sequence in $M$.   
In case $(\forall n\in \N)(\exists m\in \N)(f_{m}(0)>n)$, we obtain a sequence $(g_{n})_{n\in \N}$ with $(\forall n\in \N)\varphi(g_{n}, n)$, which contradicts our assumptions.  
Hence, we must have $(\exists n_{0}\in \N)(\forall m\in \N)(f_{m}(0)\leq n_{0})$.  Now let $n_{1}\leq n_{0}$ be such that $f_{m}=n_{1}$ for infinitely many $m\in \N$.  
We thus obtain a sub-sequence of $(f_{n})_{n\in \N}$ that is constant in $M$, and hence trivially convergent.    In conclusion, $(M, d)$ is a sequentially compact metric space that is unbounded and not countably compact, contradicting either of the items from the theorem; $\QFAC^{0,1}$ thus follows from items (c) or (d).  
\end{proof}
We note that item (b.2) formulated with second-order codes is provable in $\ACA_{0}$ by \cite{browner}*{Theorem 3.17}.
We could study \emph{total} boundedness as well.  

\smallskip

In conclusion, we have obtained Feferman's projection principle as in $\BOOT$ and countable choice as in $\QFAC^{0,1}$ from the statement \emph{compact metric spaces are separable}, i.e.\ have a second-order code.  
The weaker statement that sequentially compact spaces are \emph{bounded}, is equivalent to countable choice $\QFAC^{0,1}$, assuming the finite version of the latter.  

\subsection{Countable and measure zero sets in metric spaces}\label{BS2}
In this section, we show that basic properties of countable and measure zero sets in compact metric spaces imply $\BOOT$ and $\QFAC^{0,1}$, while also obtaining equivalences for the latter.  
We first need some definitions as follows.  
\bdefi Let $(M, d)$ be a metric space.  
\begin{itemize}
\item A set $A\subset M$ is \emph{countable} if there is $Y:M\di \N$ that is injective on $A$, i.e.\
\be\label{INJ}
  (\forall f, g \in A)( Y(f)=_{\N}Y(g)\di f=_{M}g ), 
\ee
\item A set $A\subset M$ is \emph{strongly countable} if there is a bijection $Y:M\di \N$ on $A$, i.e.\ we have \eqref{INJ} and $(\forall n\in \N)(\exists x\in A)(Y(x)=n)$; the latter means that $Y$ is surjective on $A$.
\item A set $A\subset M$ is \emph{enumerable} if there is a sequence $(x_{n})_{n\in \N}$ in $M$ that includes all the elements of $A$.
\item For an open ball $B_{d}^{M}(x, r)=\{y\in M: d(x, y)<r\}$, we put $| B_{d}^{M}(x, r)| =2r $. 
\item A set $A\subset M$ is \emph{measure zero} if for any $\eps>0$ there is a sequence of open balls  $(I_{n})_{n\in \N}$ such that $\cup_{n\in \N}I_{n}$ covers $A$ and $\eps>\sum_{n=0}^{\infty}|I_{n}|$. 
\item A set $A\subset M$ is \emph{effectively measure zero} if there is a sequence of open balls $(I_{n, k})_{n,k\in \N}$ such that for every $k\in \N$, $A\subset \cup_{n\in \N}I_{n,k}$ and $\frac{1}{2^{k}}>\sum_{n=0}^{\infty}|I_{n,k}|$. 
\end{itemize}
\edefi
The notion of `effectively measure zero' set can be found in mathematical logic and second-order RM under a different name (\cite{avi1337, nieyo}).

\smallskip

We have previously studied the RM of Tao's \emph{pigeon hole principle for measure spaces} from \cite{taoeps}*{\S1.7} in \cite{samBIG2}.
This principle expresses that the countable union of measure zero sets is also measure zero.  For codes of closed sets, this principle is provable in $\ACA_{0}$
by \cite{samBIG2}*{Theorem 3.6}.  The proof of the latter also yields that for codes of closed sets, `measure zero' implies `effectively measure zero', working in $\ACA_{0}$.
We now have the following theorem. 
\begin{thm}[$\ACAo+\SIND$]\label{NN3}
The following statements imply $\BOOT$ and $\QFAC^{0,1}$.
\begin{itemize}
\item For any metric $d:(\N^{\N}\times \N^{\N})\di \R$, a countable set in the metric space $(\N^{\N}, d)$ can be enumerated. 
\item For any metric $d:(\N^{\N}\times \N^{\N})\di \R$, a countable and closed set in the metric space $(\N^{\N}, d)$ is effectively measure zero. 
\item For any metric $d:(\N^{\N}\times \N^{\N})\di \R$, a closed and measure zero set in the metric space $(\N^{\N}, d)$ is effectively measure zero. 
\item For any metric $d:(\N^{\N}\times \N^{\N})\di \R$ and a sequence $(A_{n})_{n\in \N}$ of closed and effectively measure zero sets in $(\N^{\N}, d)$, $\cup_{n\in \N}A_{n}$ is effectively measure zero. 
\end{itemize}
The first two items are equivalent to $\BOOT+\QFAC^{0,1}$.
We can restrict to arithmetical sets and to metric spaces that are either `sequential', `limit point', `countable' or `open cover' compact.  
\end{thm}
\begin{proof}
For the first item, consider the same metric $d$ as in the proof of Theorem \ref{NN} based on $G^{2}$. 
The set $A=\{f\in \N^{\N} : f\ne_{1} 00\dots  \}$ is arithmetical and $G$ is injective on $A$.
For $(f_{n})_{n\in \N}$ an enumeration of $A$, we have that for any $n\in \N$:
\be\label{sof}
 (\exists f\in \N^{\N})(G(f)=n)\asa [(\exists m \in \N)( G(f_{m})=n )\vee G(00\dots)=n] .
\ee
Since the right-hand side is arithmetical, $\RANGE$ follows.  Moreover, $\QFAC^{0,1}$ follows in the same way as in the proof of Theorem \ref{NN}.  
For the equivalence, let $G:M\di \N$ be injective on $A\subset M$ and use $\BOOT$ to obtain $X\subset \N$ with $n\in X\asa (\exists x\in A)(G(x)=n)$ for any $n\in \N$.  
Apply $\QFAC^{0,1}$ to $(\forall n\in \N)(\exists x\in A)( n\in X\di G(x)=n)$ to obtain the required enumeration. 

\smallskip

Secondly, consider the same metric $d$ as in the proof of Theorem \ref{NN} based on $G^{2}$. 
The set $M$ is countable in the same way as in the previous paragraph.  By the second item, $M$ is effectively measure zero and let $B_{d}^{M}(f_{n, k}, r_{n, k})$ for $n, k\in \N$
be the associated sequence of open balls.  Now note that for $k_{0}\in\N $, $M$ is covered by $\cup_{n\in \N}B(f_{n,k_{0}},r_{n, k_{0}} )$, which is only possible if all $r_{n, k_{0}}$ are at most $\frac{1}{2^{k_{0}+1}}$ (as the length of an open ball is double the radius).  
As a result, the balls $B_{d}^{M}(f_{i, k_{0}}, r_{i, k_{0}})$ only contain one point of $M$ in case $G(f_{i, k_{0}})\leq k_{0}$ (and $f_{i, k_{0}}\ne 00\dots$).  Indeed, assuming $G(f_{i, k_{0}})\leq k_{0}$, the formula $\frac{1}{2^{k_{0}+1}}>d(f, f_{i, k_{0}})=|\frac{1}{2^{G(f)}}-\frac{1}{2^{G(f_{i, k_{0}})}}|$ implies $G(f)=_{0}G(f_{i, k_{0}})$ and hence $f=_{M}f_{i, k_{0}}$.    
In particular, if $(\exists f\in \N^{\N})(G(f)=k_{0})$, then there is $f_{n_{0}, k_{0}}$ such that $G(f_{n_{0}, k_{0}})=k_{0}$ for some $n_{0}\in \N$.
In particular, we obtain a version of \eqref{sof}, yielding $\BOOT+\QFAC^{0,1}$.  That the latter implies the second item is immediate given the enumeration provided by the first item.

\smallskip

Thirdly, any ball $B(00\dots, \frac{1}{2^{k+2}})$ covers all but finitely many points in $(M, d)$. 
Using $\SIND$, we observe that $M$ is measure zero; the third item yields $\BOOT+\QFAC^{0,1}$ in the same way as in the previous paragraph.

\smallskip

Fourth, consider $A$ from the first paragraph and define $A_{n}:=\{f\in \N^{\N}:  f\ne_{1}00\dots\wedge G(f)=n\}  $, which is trivially closed.  
In case $A_{n}=\emptyset$, the latter trivially has measure zero.  In case $f\in A_{n}$, then $B(f, \frac{1}{2^{k}})$ is a covering of $A_{n}$ as required for the latter being effectively measure zero. 
Then $\cup_{n\in \N}A_{n}=A $, i.e.\ the fourth item also yields $\BOOT+\QFAC^{0,1}$, following the second paragraph.
\end{proof}
A historical predecessor of the Lebesgue measure is the \emph{Peano-Jordan measure}.  The latter is connected to the Riemann integral (\cite{frink}) in the same way 
the former is connected to the Lebesgue integral.  The definition of `Peano-Jordan measure zero' amounts to `Lebesgue measure zero' where we only allow a \emph{finite} sequence of intervals.  
In light of the previous proof, Theorem \ref{NN3} also goes through for the Peano-Jordan measure and related notions.   Similarly, the Cantor-Bendixson theorem (see \cite{simpson2}*{VI.1}) or the statement \emph{a non-enumerable set has a limit point}, yield equivalences as in Theorem \ref{NN3}.

\smallskip

Finally, even restricting to strongly countable sets does not yield principles provable in $\ZF$, but nice equivalences do follow.
\begin{thm}[$\ACAo$]\label{NNX}
The following statements are equivalent.
\begin{itemize}
\item The axiom of countable choice as in $\QFAC^{0,1}$. 
\item For any metric $d:(\N^{\N}\times \N^{\N})\di \R$, a strongly countable set in the metric space $(\N^{\N}, d)$ can be enumerated. 
\item For any metric $d:(\N^{\N}\times \N^{\N})\di \R$, a strongly countable and closed set in the metric space $(\N^{\N}, d)$ is effectively measure zero. 
\end{itemize}
\end{thm}
\begin{proof}
The third item is immediate from the second item and the latter follows from the first item by applying $\QFAC^{0,1}$ to `$Y$ is surjective on $A$' in the definition of $A$ being strongly countable. 
That the third item implies $\QFAC^{0,1}$ follows from the proof of Theorem \ref{NN3}.  
\end{proof}
In conclusion, we have obtained $\BOOT$ and $\QFAC^{0,1}$ from basic statements about countable and measure zero sets.  
Weaker statements about strongly countable sets turn out to be equivalent to $\QFAC^{0,1}$.

\subsection{Continuous functions on metric spaces}\label{BSS}
\subsubsection{Introduction}\label{shintro}
In this section, we show that basic properties of (Lipschitz) continuous functions on compact metric spaces like the supremum principle, imply or are equivalent to strong axioms including $\BOOT$, $\QFAC^{0,1}$, and even full second-order arithmetic.  
We study three versions of the supremum principle for continuous $f:M\di \R$ on a compact metric space $(M, d)$. 
\begin{itemize}  
\item The supremum $\sup_{x\in \overline{B}(x_{0}, \frac{1}{2^{n}})}f(x)$ is a sequence with variable $n\in \N$.
\item The supremum $\sup_{y\in M}f(x,y)$ is a function with variable $x\in M$.
\item The supremum $\lambda f.\sup_{x\in M}f(x)$ is a functional with variable $f:M\di \R$. 
\end{itemize}
Each of these can be found in the literature (\cite{taoana2, taocompa, rudin, rudinrc, stein1}).  In particular, the supremum norm on the Banach space $C(M)$ of continuous functions, denoted 
\be\label{normy}
\|f\|_{\infty}:= \sup_{x\in M}f(x),
\ee
is essentially $\lambda f.\sup_{x\in M}f(x)$.  The coding in second-order RM allows one to define \eqref{normy} in $\RCA_{0}$ by \cite{simpson2}*{IV.2.13}.
Moreover, \emph{maximal functions} from harmonic analysis have the form $\sup_{y\in M}f(x,y)$ (see e.g.\ \cite{stein1}*{p.\ 92} or \cite{stein4}*{p.\ 198 and 208}) and occur in (generalisations of) the Hardy-Littlewood theorem (\cite{stein2}*{p.\ 48} and \cite{stein3}*{p.~246}).   Similar results and observations are possible for the extreme value theorem.  
\subsubsection{The supremum as a sequence}
We show that $\BOOT+\QFAC^{0,1}$ can be derived from a most basic supremum principle, as in Principle \ref{lipper}.
We establish nice equivalences in Theorems \ref{slab}-\ref{zirf} and Corollary \ref{stovocor}.
\begin{princ}[Supremum Principle]\label{lipper}
For any metric $d:(\N^{\N}\times \N^{\N})\di \R$ and sequentially compact $(M, d)$ with $M=\N^{\N}$, a continuous function $f:M\di \R$ has a supremum as in: $(\forall x\in M)(\exists (y_{n})_{n\in \N})(\forall n\in \N)\big(y_{n}=\sup_{z\in \overline{B}_{d}^{M}(x, \frac{1}{2^{n}})}f(z)\big)$.
\end{princ}
We also study the sequential version of the supremum principle. 
Sequential versions are studied in second-order RM in e.g.\ \cites{fuji1,fuji2,hirstseq,dork2,dork3, kooltje, simpson2,damurm,yokoyamaphd, polahirst}.  
\begin{princ}[Sequential Supremum Principle]\label{slipper}
For any metric $d:(\N^{\N}\times \N^{\N})\di \R$ and sequentially compact $(M, d)$ with $M=\N^{\N}$, let $(f_{n})_{n\in \N}$ be a sequence of continuous $M\di \R$-functions.
There is $(y_{n})_{n\in \N}$ with $(\forall n\in \N)(y_{n}=\sup_{x\in M}f_{n}(x))$.  
\end{princ}
We also study the associated extreme value theorem as follows. 
\begin{princ}[Maximum Principle]\label{lipper3}
For any metric $d:(\N^{\N}\times \N^{\N})\di \R$ and sequentially compact $(M, d)$ with $M=\N^{\N}$, a continuous function $f:M\di \R$ has a maximum as in: $(\forall x\in M)(\exists (x_{n})_{n\in \N})(\forall n\in \N)\big(f(x_{n})=\sup_{z\in \overline{B}_{d}^{M}(x, \frac{1}{2^{n}})}f(z)\big)$.
\end{princ}
First of all, we have the following theorem.  
\begin{thm}[$\ACAo+\SIND$]\label{slab} The following are equivalent.
\begin{itemize}
\item The Supremum Principle \ref{lipper} restricted to $f:M\di [0,1]$.   
\item The Sequential Supremum Principle \ref{slipper} restricted to $f:M\di [0,1]$.   
\item The projection principle $\BOOT$.
\end{itemize}
We can restrict Principle \ref{lipper} to Lipschitz continuity and we can replace `sequential' by `limit point' or `countable' or `open cover'.  
\end{thm}
\begin{proof}
First of all, consider the same metric $d$ as in the proof of Theorem~\ref{NN} based on $G^{2}$. 
Define $f:M\di \R$ as $f(x)=\frac{1}{2^{G(x)}}$ for $x\ne 00\dots$ and $f(00\dots)=0$.  Then $f$ is continuous on $M$ with modulus $H(x, k):= \frac{1}{2^{G(x)+k+1}}$.  
To establish Lipschitz continuity, note that for $x, y\ne 00\dots$, we have by definition:
\[\textstyle
|f(x)-f(y)|=|\frac{1}{2^{G(x)}}-\frac{1}{2^{G(y)}}|\leq d(x, y), 
\]
while the inequality $|f(x)-f(y)|=\frac{1}{2^{G(x)}}\leq d(x, y)$ holds if $y=00\dots$, i.e.\ $f$ is Lipschitz continuous as required.   
For $x_{0}=00\dots$, let $(y_{n})_{n\in \N}$ be a sequence such that $y_{n}=\sup_{z\in \overline{B}_{d}^{M}(x_{0}, \frac{1}{2^{n}})}f(z)$ as provided by Principle \ref{lipper}. 
By definition, we have for any $n\in \N$ that
\be\label{soft}\textstyle
(\exists f\in \N^{\N})(G(f)=n)\asa [y_{n}=_{\R}\frac{1}{2^{n}}\vee G(00\dots)=n],
\ee
and $\RANGE$ follows as the right-hand side of \eqref{soft} is arithmetical.   Regarding Principle \ref{slipper}, define $f_{n}(x)$ as $f(x)$ in case $G(x)\geq n\vee x=00\dots$, and $0$ otherwise.  
The former is continuous with modulus $H$ as above.  Moreover, let $(y_{n})_{n\in \N}$ be the sequence provided by Principle~\ref{slipper}.  
Then \eqref{soft} again holds and $\RANGE$ follows. 

\smallskip

Secondly, let $(M, d)$ be as in Principle \ref{lipper} 
and fix continuous $f:M\di [0,1]$ and $x_{0}\in M$.  Use $\BOOT$ to obtain $X_{0}\subset \N\times \Q$ such that for $n\in \N, q\in \Q\cap [0,1]$:
\[\textstyle
(n, q)\in X_{0}\asa (\exists x\in B(x_{0}, \frac{1}{2^{n}}))(f(x)>q)
\]
Clearly, the set $X_{0}$ allows one to define the supremum required by Principle \ref{lipper}.  
The modifications to obtain Principle \ref{slipper} are straightforward. 
\end{proof}
The absence of the maximum principle and the restriction to bounded functions in Theorem \ref{slab} are both essential by the following theorem. 
Items (b.2), (d.2), and (e.2) of Theorem \ref{zirf} formulated with codes are provable in $\ACA_{0}$ (see \cite{browner, simpson2, diniberg2}).
\begin{thm}[$\ACAo$]\label{zirf}
The following are equivalent.
\begin{enumerate}
\renewcommand{\theenumi}{\alph{enumi}}

\item The axiom of countable choice $\QFAC^{0,1}$.
\item The combination of the following.
\begin{itemize}
\item[(b.1)] The axiom $\QFAC^{0,1}_{\fin}$.
\item[(b.2)] For any metric $d:(\N^{\N}\times \N^{\N})\di \R$ and sequentially compact $(M, d)$ with $M\subset \N^{\N}$, any continuous $f:M\di \R$ is bounded on $M$.
\end{itemize}
\item The combination of the following.
\begin{itemize}
\item[(c.1)] The axiom $\QFAC^{0,1}_{\fin}$.
\item[(c.2)] For any metric $d:(\N^{\N}\times \N^{\N})\di \R$ and sequentially compact $(M, d)$ with infinite $M\subset \N^{\N}$, there exists a discontinuous $f:M\di \R$.  
\end{itemize}
\item The combination of the following.
\begin{itemize}
\item[(d.1)] The axiom $\QFAC^{0,1}_{\fin}$.
\item[(d.2)] \(Extreme value\) For any metric $d:(\N^{\N}\times \N^{\N})\di \R$ and sequentially compact $(M, d)$ with $M\subset \N^{\N}$, for any continuous $f:M\di \R$ with $\sup_{x\in M}f(x)$ given, there is $x_{0}\in M$ with $f(x_{0})=\sup_{x\in M}f(x)$.
\end{itemize}
\item The combination of the following.
\begin{itemize}
\item[(e.1)] The axiom $\QFAC^{0,1}_{\fin}$.
\item[(e.2)] \(Dini\).  Let $(M, d)$ be sequentially compact and let $f_{n}: (M\times \N)\di \R$ be a monotone sequence of continuous functions converging to a continuous function $f:M\di \R$.  
Then the convergence is uniform.  
\end{itemize}

\end{enumerate}
\end{thm}
\begin{proof}
First of all, we assume $\QFAC^{0,1}$ and prove the other items.  
Let $(M, d)$ be sequentially compact and let $f:M\di \R$ be continuous.  In case $f$ is unbounded on $M\subset \N^{\N}$, we have $(\forall n\in \N)(\exists x\in M)(|f(x)|> n)$.
Apply $\QFAC^{0,1}$ to obtain a sequence $(x_{n})_{n\in \N}$ with $f(x_{n})>n$ for $n\in \N$.  This sequence has a convergent sub-sequence, say with limit $z\in \N^{\N}$.  
Clearly, $f$ is not continuous at $z$, a contradiction, implying that $f$ is in fact bounded, as required for item (b).  To establish item (c), suppose additionally that $M$ is infinite, i.e.\ 
$(\forall N\in \N)( \exists x_{0}, \dots, x_{N}\in M)(\forall i, j\leq N)( x_{i}=_{M}x_{j}\di i=_{\N}j\)$.   Apply $\QFAC^{0,1}$ to obtain a sequence of pairwise different elements in $M$, called $(y_{n})_{n\in \N}$.
Now define $f:M\di \R$ as
\[
f(x):=
\begin{cases}
n & \textup{ if } (\exists n\in \N)(y_{n}=_{M}x) \\
0 & \textup{ otherwise}
\end{cases}
\]
and note that it is unbounded and therefore discontinuous by item (b), i.e.\ item~(c) also follows from $\QFAC^{0,1}$.  
To establish item (d), apply $\QFAC^{0,1}$ to $(\forall k\in \N)(\exists y\in M)(\sup_{x\in M}f(x)-\frac{1}{2^{k}} <f(y) )$.
The resulting sequence has a convergent sub-sequence, say with limit $z\in M$, and $f(z)=\sup_{x\in M}f(x)$ follows by continuity.
To establish item (e), assume the convergence is not uniform, i.e.\
\[\textstyle
(\exists k\in \N)(\forall N\in \N)(\exists x\in M )(\exists n\geq N)(|f_{n}(x)-f(x)|\geq \frac{1}{2^{k}} ).
\]
Fix such $k\in \N$ and apply $\QFAC^{0,1}$ to the resulting formula, yielding $(x_{n})_{n\in \N}$ and $g\in \N^{\N}$ with $g(n)\geq n$ and $|f_{g(n)}(x_{n})-f(x_{n})|\geq \frac{1}{2^{k}}$.
This sequence has a convergent sub-sequence, say with limit $z\in M$.  
The pointwise convergence of $(f_{n})_{n\in \N}$ at $z$ then yields a contradiction, i.e.\ $\QFAC^{0,1}$ yields item (e).  

\smallskip

Now assume item (b) and let $\varphi$ be quantifier-free such that $(\forall n\in \N)(\exists f\in \N^{\N})\varphi(f, n)$ but there is no witnessing sequence.
Consider the metric space $(M, d)$ from the proof of Theorem~\ref{NNZ} and define $F:M\di \R$ by $F(f):=f(0)$ for $f\in M$ and note that it is (vacuously) continuous on $M$.  
Using $\QFAC^{0,1}_{\fin}$, $F$ is unbounded on $M$, contradiction, and $\QFAC^{0,1}$ follows.  Moreover, since any function $H:M\di \R$ is (vacuously) continuous, item~(c) also implies $\QFAC^{0,1}$.
To show that item (d) implies $\QFAC^{0,1}$, define $G(f):= 1-\frac{1}{2^{f(0)}}$ and note that $G$ is continuous on $(M, d)$ as above with $\sup_{f\in M}G(f)=1$, using $\QFAC^{0,1}_{\fin}$.
Clearly, there is no $f\in M$ such that $G(f)=1$.  To show that item~(e) implies $\QFAC^{0,1}$, define $F_{n}(f):= f(0)$ if $f(0)\leq n$ and $0$ otherwise.  
Clearly, $(F_{n})_{n\in \N}$ converges pointwise to $F(f):=f(0)$ and continuity on $(M, d)$ as above is again straightforward.  However, the convergence cannot be uniform, and item (e) must imply $\QFAC^{0,1}$.  
\end{proof}
Due to the `sequential' nature of the supremum in Principle \ref{lipper}, it is not that surprising that one obtains $\QFAC^{0,1}$.
However, Theorem \ref{zirf} only requires an upper bound, a concept free of any `sequential-ness'.  
Following Theorem \ref{NNZ}, we could replace `bounded' in Theorem \ref{zirf} by `uniformly continuous'.  

\smallskip

Finally, the previous results combine into the following elegant summary.
\begin{cor}[$\ACAo+\SIND$]\label{stovocor}
The following are equivalent.
\begin{itemize}
\item $\BOOT+\QFAC^{0,1}$
\item The axiom $\QFAC^{0,1}_{\fin}$ and the Sequential Supremum Principle \ref{slipper}. 
\end{itemize}
\end{cor}
As to related results, the \emph{Stone-Weierstrass theorem} is studied in second-order RM in e.g.\ \cite{browner}*{\S4}; results similar to Corollary \ref{stovocor} are possible, but not as elegant. 
Similarly, $\BOOT$ can be derived from the existence of an inverse for bi-Lipschitz functions on compact metric spaces.  Moreover, the \emph{intermediate value theorem} for chain connected and compact $(M, d)$ with $M\subset \N^{\N}$, is equivalent to $\QFAC^{0,1}$.   
We shall establish a more general result in Section \ref{zonggg}

\smallskip

In conclusion, the basic supremum principle as in Principle \ref{lipper} is equivalent to $\BOOT+\QFAC^{0,1}$ with equivalences for the individual principles as well.
By contrast, Principle \ref{lipper} formulated with second-order codes is provable in $\ACA_{0}$, as noted above.  
One could obtain analogous results for the existence of the distance function $d(x, C):= \inf_{y\in C}d(x, y)$, but the latter is an infimum anyway. 

\subsubsection{The supremum as a function}
In this section, we study a basic supremum principle on $\NN\equiv \N^{\N}\times \N^{\N}$ as in Principle \ref{lipper2}.  Note that we assume the existence of $\lambda x. \sup_{y\in \N^{\N}}f(x, y)$ as a function from Baire space to the reals. 
This notation is found throughout textbooks and the literature, as discussed in Section \ref{shintro}
\begin{princ}\label{lipper2}
 For any metric $d:\NN^{2}\di \R$ and sequentially compact metric space $(\NN, d)$, 
and Lipschitz continuous $f:\NN\di \R$, the supremum $\sup_{y\in \N^{\N}}f(x, y)$ and infimum $\inf_{y\in \N^{\N}}f(x, y)$ exist\footnote{To be absolutely clear, we assume that there is $\Phi:\N^{\N}\di (\N^{\N}\times \N^{\N})$ such that $\Phi(x)(1)=\sup_{y\in \N^{\N}}f(x, y) $ and $ \Phi(x)(2)=\inf_{y\in \N^{\N}}f(x, y)$ for any $x\in \N^{\N}$.} for any $x\in \N^{\N}$.
\end{princ}
We first show that Principle \ref{lipper2} implies the following generalisation of $\BOOT$.
\begin{princ}[$\BOOT_{2}$] For any $Y:(\N^{\N}\times \N^{\N})\di \N$, there is $ X\subset \N$ with 
\[
(\forall n\in \N  )[ n\in X\asa (\exists f\in \N^{\N})(\forall g\in \N^{\N})(Y(f,g, n)=0) ].
\]
\end{princ}
\noindent
We note that $\BOOT_{2}\di \FIVE$ over $\RCAo$ while $\ACAo+\BOOT_{2}$ proves $\SIX$.  
We now have the following implication.
\begin{thm}[$\ACAo+\SIND$]\label{NN2}
The supremum principle as in Principle \ref{lipper2} implies $\BOOT_{2}$.
We can restrict to Lipschitz continuity and we can replace `sequential' by `limit point' or `countable' or `open cover'.  
\end{thm}
\begin{proof}
First of all, in the same way as for $\BOOT\asa \RANGE$ in the proof of Theorem~\ref{NN}, the principle $\BOOT_{2}$ is equivalent to the following
\[
(\forall G:\NN\di \N)(\exists X\subset \N)(\forall n \in \N)\big[n\in X\asa (\exists f\in \N^{\N})(\forall g\in \N^{\N})(G(f, g)=n)  ].
\]
We will use `$\overline{h}$' to denote elements of $\NN$ and write $G(\overline{h})=G(f, g)$ in case $\overline{h}=(f, g)$, and where $\overline{0}:=(00\dots, 00\dots)$.

\smallskip

Now fix $G:\NN\di \N$ and define $d:(\NN\times \NN)\di \R$ as follows:  $d(\overline{f}, \overline{g}):=|\frac{1}{2^{G(\overline{f})}}-\frac{1}{2^{G(\overline{g})}}|$ for $\overline{f}, \overline{g}\ne \overline{0}$, 
$d(\overline{0}, \overline{f})=d(\overline{f}, \overline{0}):= \frac{1}{2^{G(\overline{f})}} $ for $\overline{f}\ne \overline{0}$, and $d(\overline{0}, \overline{0}):=0$. 
Then $(M, d)$ is a metric space for $M=\NN$ if we define the equivalence relation `$=_{M}$' as:
\[
\overline{f}=_{M}\overline{g}\equiv\big[~ [\overline{f},\overline{g}\ne \overline{0} \wedge G(\overline{f})=G(\overline{g})] \vee \overline{f}=\overline{g}=\overline{0}~\big].  
\]
Indeed, the first two items of Definition \ref{donkc} hold by definition while for the third item, observe that for $\overline{f},\overline{ g}\ne \overline{0}$, we have 
\[\textstyle
d(\overline{f}, \overline{g})=|\frac{1}{2^{G(\overline{f})}}-\frac{1}{2^{G(\overline{g})}}|\leq |\frac{1}{2^{G(\overline{f})}}-\frac{1}{2^{G(\overline{h})}}|+ |\frac{1}{2^{G(\overline{h})}}-\frac{1}{2^{G(\overline{g})}}| = d(\overline{f}, \overline{h})+d(\overline{h}, \overline{g}), 
\]
for $\overline{h}\ne \overline{0}$;  also $d(\overline{f},\overline{ g})=|\frac{1}{2^{G(\overline{f})}}-\frac{1}{2^{G(\overline{g})}}|\leq \frac{1}{2^{G(\overline{f})}}+\frac{1}{2^{G(\overline{g})}} = d(\overline{f}, \overline{0})+d(\overline{0}, \overline{g})$.  
Moreover, $d(\overline{f}, \overline{0})=\frac{1}{2^{G(\overline{f})}}\leq \frac{1}{2^{G(\overline{g})}} + |\frac{1}{2^{G(\overline{f})}}-\frac{1}{2^{G(\overline{g})}}| =d(\overline{g}, \overline{0})+d(\overline{f}, \overline{g})$ for $\overline{g}\ne \overline{0}\dots$; also $0=d(\overline{0}, \overline{0})\leq d(\overline{0}, \overline{f})+d(\overline{f}, \overline{0})$ for any $\overline{f}\in \NN$.

\smallskip

To show that $(M, d)$ is sequentially compact, let $(\overline{f}_{n})_{n\in \N}$ be a sequence in $M$.  In case $\overline{0}$ occurs infinitely many times or if $(\exists m\in \N)(\forall n\in \N)(G(\overline{f}_{n})\leq m)$, there is a trivial constant sub-sequence.  
In case $(\forall m\in \N)(\exists n\in \N)(G(\overline{f}_{n})> m)$, let $(\overline{g}_{n})_{n\in \N}$ be a sub-sequence with $G(\overline{g}_{n})>n$ for all $n\in \N$.  Clearly, $(\overline{g}_{n})_{n\in \N}$ converges to $\overline{0}$ in $(M, d)$.
To show that $(M, d)$ is (countably) compact, note that any open ball containing $\overline{0}$ covers all but finitely elements of $M$. 
We use $\SIND$ to enumerate the associated finite set.  
Similarly, for any infinite set $X\subset M$, $\overline{0} $ is a limit point of $X$.

\smallskip

Finally, consider $F:\NN\di \R$ defined as $F(\overline{f}):=\frac{1}{2^{G(\overline{f})}}$ for $\overline{f}\ne \overline{0}$ and $F(\overline{0})=0$.
By the definition of the metric, $F$ is Lipschitz.  Now consider, for any $n\in \N$:
\begin{align*}\textstyle
& (\exists f\in \N^{\N})(\forall g\in \N^{\N})(G(f, g)=n) \\
 &\textstyle \asa(\exists f\ne 00\dots)(\forall g \in \N^{\N})(F(f, g)=\frac{1}{2^{n}}) \vee (\forall g\in \N^{\N})(G(00\dots, g)=n)  \\
 &\textstyle \asa (\exists f\ne 00\dots)( \frac{1}{2^{n}}=\inf_{g \in \N^{\N} } F(f, g)=\sup_{h\in \N^{\N}}F(f, h) ) \vee (\forall g\in \N^{\N})(G(00\dots, g)=n).
\end{align*}
Applying $\BOOT$ (provided by Theorem \ref{slab}) to each disjunct of the final formula, we obtain $\BOOT_{2}$.
\end{proof}
Let $\BOOT_{k}$ for $k\geq 3$ be the obvious generalisation of $\BOOT_{2}$ to $k-1$ quantifier alternations.  
We note that $\BOOT_{k}\di \SIXk$ over $\RCAo$ while $\ACAo+\BOOT_{k}$ proves $\SIXko$.  
We now have the following corollary.
\begin{cor}[$\ACAo+\SIND$]\label{NNcozz}
The supremum principle as in Principle \ref{lipper2} implies $\BOOT_{k}$ for $k\geq 2$.
We can replace `sequential' by `limit point' or `countable' or `open cover'.  
\end{cor}
\begin{proof}
We prove $\BOOT_{3}$ from Principle \eqref{lipper2}; the general case is then straightforward. 
Fix $H:(\N^{\N}\times \N^{\N}\times \N^{\N})\di \N$ and define $G(f, g):= H(P(f)(1), P(f)(2), g)$ where $f=_{1}P(f)(1)\oplus P(f)(2)$.
We define $(M, d)$ as in the proof of Theorem \ref{NN2} using $G:\NN\di \N$; the former is again a sequentially compact metric space as in the proof of Theorem \ref{NN2}.  
The function $F:\NN\di \R$ from the proof of Theorem~\ref{NN2} is again (Lipschitz) continuous.
Let $A(n)$ be the formula $(\exists g \in \N^{\N})(\forall h\in \N^{\N})(H(00\dots , g, h)=n)$ and consider the following, for any $n\in \N$:
\begin{align*}\textstyle
& (\forall f\in \N^{\N})(\exists g\in \N^{\N})(\forall h\in \N^{\N})(H(f, g, h)=n) \\
 &\textstyle \asa(\forall f\ne_{1}00\dots)(\exists g \in \N^{\N})(\forall h\in \N^{\N})(F((f\oplus g, h))=\frac{1}{2^{n}}) \vee A(n)  \\ 
 &\textstyle \asa (\forall  f\ne_{1}00\dots)(\exists g\in \N^{\N})\big[ \frac{1}{2^{n}}=\inf_{h \in \N^{\N} } F((f\oplus g, h))=\sup_{j\in \N^{\N}}F((f\oplus g, j)) \big]\vee A(n).
\end{align*}
Applying $\BOOT_{2}$ to both disjuncts in the final formula, we obtain $\BOOT_{3}$.
\end{proof}
The \emph{maximum principle} expresses that the supremum of a continuous function is attained on a compact space.  
The following principle expresses that $\max_{y\in M}f(x, y)$ and $\min_{y\in M}f(x, y)$ exist as $M\di M$-functions.  
The latter operators are found in \emph{continuous optimisation} with examples in \cite{coop} or \cite{plopp}*{p.\ 160, 173, and 194}.
\begin{princ}\label{lipper5}
 For any metric $d:\NN^{2}\di \R$ and sequentially compact metric space $(\NN, d)$, 
and continuous $f:\NN\di \R$, there is $\Phi:\NN\di \NN^{2}$ such that for all $x\in \NN$, 
\[
(\forall y\in M) \big[ f(x, \Phi(x)(1))\leq f(x, y) \leq_{\R} f(x, \Phi(x)(2))\big].
\]
\end{princ}
\noindent
One readily derives $\BOOT_{k}$ from the previous principle. 

\smallskip

In conclusion, the basic supremum principle as in Principle \ref{lipper2} implies $\BOOT_{k}$ and $\SIXk$ for $k\geq 3$.
By contrast, Principle \ref{lipper2} formulated with second-order codes is provable in $\ACA_{0}^{\omega}$ as $(\exists^{2})$ can define the supremum operator via the usual interval-halving technique; the code $A$ of the (complete separable) metric space $\hat{A}$ 
guarantees that quantifying over $\hat{A}$ can be replaced by quantifying over $A$, which is arithmetical by definition.

\subsubsection{The supremum as a functional}
In this section, we study a supremum principle stating the existence of a supremum functional $\lambda f. \sup_{x\in M}f(x)$ for continuous $M\di\R$-functions.
The latter amounts to the supremum norm $\|f\|_{\infty}$ on $C(M)$ as in \eqref{normy}.
In light of \cite{browner}*{Theorem 4.1}, $\ACAo$ proves the existence of such a supremum functional on compact metric spaces that come with a code.  
By contrast, $\Z_{2}^{\Omega}$, a conservative extension of $\Z_{2}$, is needed for the supremum functional.   
\begin{thm}[$\RCAo+\SIND$]\label{slabz}
The following are equivalent. 
\begin{enumerate}
\renewcommand{\theenumi}{\alph{enumi}}
\item Kleene's quantifier $(\exists^{3})$. 
\item The combination of the following.
\begin{itemize}
\item[(b.1)] Kleene's quantifier $(\exists^{2})$. 
\item[(b.2)] For any metric $d:(\N^{\N}\times \N^{\N})\di \R$ and sequentially compact $(M, d)$ with $M=\N^{\N}$, there is $\Phi:(M\di \R)\di \R$ such that for any continuous $f:M\di [0,1]$, we have $\Phi(f)=\sup_{x\in M}f(x)$.  
\end{itemize}
\end{enumerate}
We can restrict to Lipschitz continuity and replace `sequential' by `limit point' or `countable' or `open cover'.  
\end{thm} 
\begin{proof}
The `downward' implication is straightforward as the usual interval-halving technique using $(\exists^{3})$ readily yields $\sup_{x\in M}f(x)$ for bounded $f:M\di \R$.  
For the upward implication, fix $G:\N^{\N}\di \N$ and consider the metric space $(  M, d)$ from the proof of Theorem \ref{slab}. 
Let $f:M\di \R$ be the function $f(x):=\frac{1}{2^{G(x)}}$ for $x\ne_{M}00\dots$ and $f(00\dots)=0$, which is Lipschitz continuous as in the proof of the latter.  Now note that 
\be\textstyle\label{zenda}
(\exists g\in \N^{\N})(G(g)=0)\asa [\sup_{x\in \N^{\N}}f(x)=_{\R}\frac{1}{2}   \vee G(00\dots)=0] .
\ee
The right-hand-side of \eqref{zenda} is decidable using $(\exists^{2})$, i.e.\ $(\exists^{3})$ follows from item (b).  
The final sentence is immediate in light of Theorem \ref{zirf} and its proof.  
\end{proof}
The restriction to bounded functions in Theorem \ref{slabz} can be lifted, as follows.
\begin{cor}
The theorem holds for the generalisation to $f:M\di \R$ and $M\subset \N^{\N}$ in item \textup{(b.2)}, if we add $\QFAC^{0,1}$ to item \textup{(a)} and $\QFAC^{0,1}_{\fin}$ to item \textup{(b)}. 
\end{cor}
\begin{proof}
The corollary follows by Theorem \ref{zirf}. 
\end{proof}
In conclusion, the supremum principle as in item (b.2) in Theorem~\ref{slabz} is equivalent to $(\exists^{3})$.
In light of \cite{browner}*{Theorem 4.1}, $\ACAo$ proves the existence of such a supremum functional on compact metric spaces that come with a code.  

\subsection{Second-countability and metric spaces}\label{sleepwithoneeye}
We study the statement 
\begin{center}
\emph{compact metric spaces are second-countable, i.e.\ have a countable base}
\end{center}
and show that it \emph{can} be quite strong, implying Feferman's projection principle and countable choice. 

\smallskip

First of all, the strength of the previous centred statement depends on which definition of \emph{countable base} one uses (compare \cite{munkies}*{p.\ 78} and \cite{bengelkoning}*{p.~12}).  
For this reason, we adopt Definition \ref{SCC} where SC1 corresponds to the coding of open sets in second-order RM (\cite{simpson2}*{II.5.6}) and SC2 corresponds to Dorais' CSC-spaces from second-order RM (\cite{damurm}*{\S10.8.1}).
\bdefi[Second-countability]\label{SCC} 
Let $(M, d)$ be a metric space.  We say that:   
\begin{itemize}
\item[(SC1)] $(M, d)$ is SC1 in case there is a sequence of open sets $(O_{n})_{n\in \N}$ such that for every open $O\subset M$, there is $g\in \N^{\N}$ with $O=\cup_{n\in \N}O_{g(n)}$,
\item[(SC2)] $(M, d)$ is SC2 in case there is a sequence of open sets $(O_{n})_{n\in \N}$ satisfying: 
\begin{itemize}
\item for any $x\in M$, there is $n\in \N$ with $x\in O_{n}$,
\item for any $x\in M$, if $x\in O_{n}\cap O_{m}$, there is $k\in \N$ with $x\in O_{k}\subset O_{n}\cap O_{m}$.
\end{itemize}
\end{itemize}
\end{defi}
By Theorem \ref{SCS1}, $\BOOT+\QFAC^{0,1}$ follows from the basic statement \emph{a compact metric space is SC1} and special cases.  
By contrast, Theorem \ref{zibk} shows that many of the above results still go through if we restrict to SC2-spaces. 

\smallskip

Secondly, the following principle expresses that sequentially compact metric spaces are SC1.  
The associated principle for open sets of reals is studied in \cite{dagsamVII}, where it is conjectured to be weaker than $\BOOT$. 
\begin{princ}[$\open_{0}$]\label{lack0}
For any metric $d:(\N^{\N}\times \N^{\N})\di \R$ and sequentially compact $(M, d)$ with $M\subseteq\N^{\N}$, and for any open set $O\subset M$, there are sequences $(x_{n})_{n\in \N}$ and $(r_{n})_{n\in \N}$ such that $O=\cup_{n\in \N}B_{d}^{M}(x_{n}, r_{n})$.  
\end{princ}
The countable union in $\open_{0}$ amounts to the representation of open sets in second-order RM as in \cite{simpson2}*{II.5.6}.  
Now, the latter representation can be effectively converted into another representation based on continuous functions involving a continuous modulus (see \cite{simpson2}*{II.7.1} and \cite{kohlenbach4}*{Theorem 4.4}).
Inspired by this observation, we let $\open_{1}$ be $\open_{0}$ restricted to open $O\subset M$ for which there is a continuous $h:M\di \R$ with a continuous modulus such that $x\in O\asa h(x)>0$ for any $x\in M$.
Moreover, we have studied a different coding in \cite{dagsamVII}*{\S6-7}, as follows.  
\bdefi[R2-representation]
An open set $O\subset M$ in a metric space $(M, d)$ has an \emph{R2-representation} if there is $\Psi:M\di \R$ such that $x\in O\asa \Psi(x)>0$ and $x\in O\di B_{d}^{M}(x, \Psi(x))\subset O$ for any $x\in M$.
\edefi
Let $\open_{2}$ be Principle \ref{lack0} restricted to open sets with an R2-representation. 
The principle $\open_{2}$ for the real line is rather weak in light of \cite{dagsamVII}*{\S6-7}.
Despite the differences among the principles $\open_{i}$, we have the following theorem.
\begin{thm}[$\ACAo+\SIND$]\label{SCS1}
For $i=0, 1,2$, $\open_{i}$ implies $\BOOT+\QFAC^{0,1}$.
We may replace `sequentially' by `limit point' or `countably' in the former principles. 
\end{thm}
\begin{proof}
Fix $G:\N^{\N}\di \N$ and let $(q_{n})_{n\in \N}$ be an enumeration of the set of rationals $\Q$ without repetitions. We also let $0_{M}$ be a new symbol, different by fiat from any rational or element of Baire space.  
Up to coding, we define $M$ as the union of $\{0_{M}\}$, $\Q$, and $\N^{\N}$.  For $x\in M$, define $H(0_{M})=0$, $H(x):= 2n+1$ if $x=q_{n}\in \Q$, and $H(x)=2G(x)$ if $x\in \N^{\N}$.
Now define $d:M^{2}\di \R$ as follows for $x, y\in M$:  $d(0_{M}, 0_{M})=0$, $d(x, 0_{M})=\frac{1}{2^{H(x)}}$ for $x\ne 0_{M}$, and $d(x, y):= |\frac{1}{2^{H(x)}}-\frac{1}{2^{H(y)}}|$ for $x, y\ne 0_{M}$. 
As in the above proofs, $(M, d)$ is a metric space that is sequentially, limit point, and (countably) compact.  

\smallskip

Define the set $C\subset M$ as the union of $\{0_{M}\}\cup \Q$.  Taking convergence to be within $(M, d)$, we note that a sequence with elements in $C$ that converges, must converge to $0_{M}$ or be eventually constant, i.e.\ $C$ is closed.
Define the open set $O:=M\setminus C$ and let $(x_{n})_{n\in \N}$ and $(r_{n})_{n\in \N}$ be such that $O=\cup_{n\in \N}B_{d}^{M}(x_{n}, r_{n})$, as provided by $\open_{0}$.   
Now consider the following
\[
(\exists x\in \N^{\N})(G(x)=n)\asa  [(\exists i\in \N)( G(x_{i})=n) \vee G(00\dots)=n],
\]
which follows by the definition of $(M, d)$.  Thus, $\RANGE$ and $\QFAC^{0,1}$ follow from $\open_{0}$ in the same way as in the proof of Theorem \ref{NN2}.  
To obtain the same result for $\open_{2}$, define $\Psi:M\di \R$ as follows:  $\Psi(x)=\frac{1}{2^{H(x)+5}}$ if $x\in O$ and $\Psi(x)=0$ otherwise. 
This constitutes an R2-representation of $O$, as required for $\open_{2}$.  
Indeed, we have $x\in O\asa \Psi(x)>0$ by definition, while $x\in O$ implies $\{x\}=B(x, \Psi(x))\subset O$.  
For $\open_{1}$, we observe that $\Psi$ is continuous on $M$ with modulus $\Phi(x, k)=\frac{1}{2^{H(x)+k+5}}$, which in turn is continuous in $x$ on $M$ for any $k\in \N$.
This again follows from the fact that $d(x, y)<\frac{1}{2^{H(x)+5}}$ implies $x=_{M}y$.
\end{proof}
The following generalisation of $\open_{0}$ implies $\BOOT_{k}$ for $k\geq 3$, following the proof of Corollary \ref{NNcozz}.
Note that for an open set $O\subset M\subset \NN$, the projection $O_{f}:=\{ g\in \N^{\N}: (f, g)\in O   \}$ is also open. 
\begin{princ}\label{lipperk}
 For any metric $d:\NN^{2}\di \R$ and sequentially compact metric space $(M, d)$ with $M\subset \NN$, there exists $\Phi$ such that for any $f\in \N^{\N}$ and open $O\subset M$ $\Phi(C, f)$ is a code for the projection $O_{f}$. 
\end{princ}
Finally, we show that some of our results still go through for SC2-spaces.  
\begin{thm}[$\ACAo$]\label{zibk}
Theorem \ref{slab} goes through restricted to SC2-spaces. 
\end{thm}
\begin{proof}
The proof of Theorem \ref{slab} makes use of the metric space $(M, d)$ defined above \eqref{bookana} based on $M=\N^{\N}$ and $G^{2}$.  
It suffices to prove that this metric space is SC2.  
To this end, let $(O_{n})_{n\in \N}$ be the following sequence of open sets: $O_{2n}$ is $\{x\in M:  G(x)=n\}$ while $O_{2n+1}$ is $B_{d}^{M}(00\dots, \frac{1}{2^{n}})$. 
This sequence forms a countable base as in SC2, as $(\forall x\in M)(\exists n\in \N)(x\in O_{n})$ is trivial.  In case $x\in O_{n_{1}}\cap O_{n_{2}}$, define $k:(M\times \N^{2})\di \N$ as follows: $k(x, n_{1}, n_{2})=\max(n_{1}, n_{2})$ if $n_{1}, n_{2}$ are odd and define $k(x, n_{1}, n_{2})=n_{1}$ in case $n_{1}$ is even, and $n_{2}$ otherwise.  By definition, we have $x\in O_{k(x, n_{1}, n_{2})}\subset O_{n_{1}}\cap O_{n_{2}}$, as required.     
\end{proof}
\noindent
Most results in this paper seem to go through with the restriction to SC2-spaces in place.  We note that the previous proof actually establishes the existence of a \emph{strong} countable base as in \cite{damurm}*{Def.\ 10.8.2}.  

\smallskip
\noindent
We finish this section with remarks on variations of the above results.  
\begin{rem}[Located sets]\rm
In second-order RM, a closed set $C$ is called \emph{located} if $d(x, C)=\inf_{y\in C}d(x, y)$ exists as a (code for a) continuous function.   
By \cite{browner}*{Theorems~3.10 and 3.17}, $\ACA_{0}$ suffices to show that, formulated using codes closed sets are located in sequentially compact metric spaces. 
By contrast, one readily shows that $\BOOT$ follows from the statement that for any metric space $(M, d)$ with $M\subset \N^{\N}$, every closed $C\subset M$, there exists $\Phi:M\di \R$ such that  $\Phi(x)=d(x, C)$ for any $x\in M$.
Moreover, $\BOOT_{k}$ follows from the existence of a distance function in $\NN$ of the form $\lambda x, y.d(x, C_{y})$ where $C_{y}=\{z\in M: (y, z) \in C \}$.
\end{rem}
\begin{rem}[Urysohn and Tietze]\rm
The `original' principle $\open$ from \cite{dagsamVII} is the statement that an open set $O\subset \R$ has an RM-code as in \cite{simpson2}*{II.5.6}. 
As proved in the former, $\open$ is equivalent to the associated versions of the Urysohn lemma and the Tietze extension theorem.  
In the context of compact metric spaces, it seems that the latter are (much) weaker than $\open_{0}$.  This is no surprise as continuous functions are not `elementary' objects 
in the absence of a countable dense subset, by the above.    
\end{rem}
In conclusion, basic statements about open sets in compact metric spaces, like $\open_{i}$ for $i=0,1 ,2$, imply $\BOOT+\QFAC^{0,1}$. 
The former merely express that open sets come with second-order codes as in the definition of SC1-space.  By contrast, Theorem \ref{zibk} implies that the restriction to SC2-spaces is mostly inconsequential.

\section{Beyond the Baire space}\label{zonggg}
\subsection{Introduction}
In this section, we study compact metric spaces $(M, d)$ with the restriction `$M\subset \N^{\N}$' from Section \ref{main} lifted, i.e.\ $M$ and $d:M^{2}\di \R$ can be fourth-order and beyond.
As we will see, the proofs from Section \ref{main} provide a template for the proofs in this section.  
The associated fourth- and higher-order theorems, like the supremum principle, are \emph{much} stronger than in the third-order case, but not as natural anymore, as discussed next in detail.  

\smallskip

\noindent
First of all, our motivation for dropping the restriction `$M\subset \N^{\N}$' is as follows. 
\begin{itemize}
\item For a compact metric space $(M, d)$, the associated compact metric space $C(M)$ of continuous $M\di \R$-functions, is a central topic of study.  
The space $C(M)$ evidently includes objects of higher type than $M$.
\item For any compact metric space $(M, d)$, the set $M$ has at most the cardinality of the continuum (\cite{buko}*{Theorem~3.13}), i.e.\ we do 
not increase the size of $M$ by dropping the restriction `$M\subset \N^{\N}$'.  
\item The intuitive idea of metric spaces is that one studies `any set with a notion of distance'.  No restriction on the type of the objects involved is assumed.  
Moreover, the Lebesgue spaces $L^{p}$ or Skorokhod space $\mathcal{D}$ consist of possibly discontinuous functions.  
\end{itemize}
Despite these observations, we believe that compact metric spaces $(M, d)$ with $M$ containing all $\N^{\N}\di \N^{\N}$-mappings are much less natural than
the spaces from Section~\ref{main} for which $M\subset \N^{\N}$, or the spaces mentioned in the above items.

\smallskip

Secondly, we sometimes use type-theoretic notation, like $n^{0}$ for natural numbers $n\in \N$, $f^{1}$ for elements of Baire space $f\in \N^{\N}$, $Y^{1\di 1}$ for $\N^{\N}\di \N^{\N}$-mappings, and $Z^{3}$ for mappings from $\N^{\N}\di \N$ to $\N$.  The system $\RCAo$ is fundamentally based on the language of all finite types, i.e.\ nothing new has to be introduced.  Moreover, $\RCAo$ only has very basic axioms governing fourth-order and higher objects.  We shall use $\NNN$ to (symbolically) denote the collection of all $\N^{\N}\di \N^{\N}$-mappings.    

\smallskip

Thirdly, with the above notation, we can consider following generalisation of Feferman's projection principle. 
\begin{princ}[$\BOOT_{1}^{2}$]\label{fren}
$(\forall Z^{3})(\exists f^{1})(\forall n^{0}  )[ f(n)=0\asa (\exists Y\in \NNN)(Z(Y, n)=0) ].$
\end{princ}
We observe that $\Z_{2}^{\Omega}+\BOOT_{1}^{2}$ proves comprehension for $\Pi_{1}^{2}$-formulas. 
We shall study the following statements about {sequentially} compact metric spaces $(M, d)$ with $M\subset \NNN= \N^{\N}\di \N^{\N}$ and indicate generalisations to all finite types.
\begin{itemize}
\item[(a)] The supremum principle for continuous functions on $(M, d)$ is equivalent to $\BOOT_{1}^{2}$ (Section \ref{compr}) over a relatively strong base theory.
\item[(b)] The intermediate value theorem for chain connected $(M, d)$ is equivalent to the Axiom of countable\footnote{The axiom $\QFAC^{0,\NNN}$ is just $\QFAC^{0, 1\di 1}$ as follows: for a quantifier-free formula $\varphi$ with $(\forall n^{0})(\exists Y^{1\di 1})\varphi(Y, n)$, we have $(\exists (Y_{n})_{n\in \N})(\forall n^{0})\varphi(Y_{n}, n)$.} Choice as in $\QFAC^{0, \NNN}$  (Section~\ref{boundz}).
\end{itemize}
We often only sketch the proofs in this section as they are a straightforward modification of the proofs in Section~\ref{main}.
The previous results suggest another motivation for studying e.g.\ countable choice in higher-order RM: the resulting theorems are fine-grained and provide an elegant hierarchy.

\subsection{Comprehension}\label{compr}
We show that the supremum principle for continuous functions on compact metric spaces $(M, d)$ with $M=\NNN$ as in Principle \ref{lipperz} is equivalent to $\BOOT^{2}_{1}$ (Principle~\ref{fren}).
Recall that $\NNN$ is the set of mappings from Baire space to itself, i.e.\ the following principle is fourth-order.  
\begin{princ}[Supremum Principle]\label{lipperz}
For any metric $d:\NNN^{2}\di \R$ and sequentially compact $(M, d)$ with $M=\NNN$, a continuous function $f:M\di [0,1]$ has a supremum as in: $(\forall x\in M)(\exists (y_{n})_{n\in \N})(\forall n\in \N)(y_{n}=\sup_{z\in \overline{B}_{d}^{M}(x, \frac{1}{2^{n}})}f(z))$.
\end{princ}
\begin{thm}[$\Z_{2}^{\Omega}$]\label{slabllll} The following are equivalent.
\begin{itemize}
\item The Supremum Principle \ref{lipperz} restricted to $f:M\di [0,1]$.   
\item The projection principle $\BOOT_{1}^{2}$.
\end{itemize}
We can restrict to Lipschitz continuity and replace `sequential' by `limit point'.  
\end{thm}
\begin{proof}
The principle $\BOOT_{1}^{2}$ is equivalent to $\RANGE_{2}$ as follows:
\be\label{myhunter}\tag{$\RANGE_{2}$}
(\forall G:\NNN\di \N)(\exists X\subset \N)(\forall n \in \N)\big[n\in X\asa (\exists Y\in \NNN)(G(Y)=n)  ].
\ee
Indeed, for the reverse direction, fix $Z: (\NNN\times \N)\di \N$ and define $G:\NNN\di \N$ as follows for $n^{0}$ and $Y^{1\di 1}$: put $G( Y )=n+1$ if $Z(\lambda f.\lambda m.Y(f)(2m+1), n)=0$ for $2n=Y(00\dots)(0)$, and $0$ otherwise. 
Let $X\subseteq \N$ be as in $\RANGE_{2}$ and note that 
\[
(\forall m \geq 1 )( m\in X \asa (\exists Y\in \NNN)(G(Y)=m)\asa (\exists W\in \NNN)(Z(W, m-1)=0)  ),
\]
which is as required for $\BOOT_{1}^{2}$ after some trivial modification. 

\smallskip

Now fix $G:\NNN\di \N$, let $o^{1\di 1}$ be such that $o(f)=00\dots$ for all $f^{1}$, and define $\NNN_{0}:= \NNN\setminus \{o \}$ using $(\exists^{3})$.  Define $d:\NNN^{2}\di \R$ as $d(Y, W):=|\frac{1}{2^{G(Y)}}-\frac{1}{2^{G(W)}}|$ for $Y, W \in \NNN_{0}$, $d(o, Y)=d( Y,o):= \frac{1}{2^{G(Y)}} $ for $Y\in \NNN_{0}$, and $d(o, o):=0$. 
Then $(M, d)$ is a metric space for $M=\NNN$ if we define the equivalence relation `$=_{M}$' as follows:
\[
Y=_{M}W\equiv\big[ [Y, W\in \NNN_{0} \wedge G(W)=G(Y)] \vee Y=_{1\di 1}W=_{1\di 1}o\big], 
\]
using $(\exists^{3})$.
As in the proof of Theorem \ref{slab}, $(M, d)$ is a sequentially/limit point compact metric space.  
Again using $(\exists^{3})$, define $F:M\di \R$ as $F(Y)=\frac{1}{2^{G(Y)}}$ for $Y\in \NNN_{0}$ and $F(o)=0$, which is Lipschitz continuous as in the proof of Theorem~\ref{slab}.   
For $x_{0}=o$, let $(y_{n})_{n\in \N}$ be a sequence such that $y_{n}=\sup_{z\in \overline{B}_{d}^{M}(x_{0}, \frac{1}{2^{n}})}F(z)$ as provided by Principle \ref{lipperz}. 
By definition, we have for any $n\in \N$ that
\be\label{soft2}\textstyle
(\exists Y\in \NNN)(G(Y)=n)\asa [y_{n}=_{\R}\frac{1}{2^{n}} \vee G(o)=n],
\ee
and $\RANGE_{2}$ follows as the right-hand side of \eqref{soft2} is arithmetical.   

\smallskip

Secondly, 
 fix a continuous $F:M\di [0,1]$ and $x_{0}\in M$ and use $\BOOT_{1}^{2}$ to obtain $X_{0}\subset \N^{2}$ such that for all $n, k\in \N$
\[\textstyle
(n, k)\in X_{0}\asa (\exists x\in B(x_{0}, \frac{1}{2^{n}}))(F(x)>1-\frac{1}{2^{k}})
\]
Clearly, the set $X_{0}$ allows one to define the supremum required by Principle \ref{lipperz}.  
\end{proof}
The previous proof amounts to the proof of Theorem \ref{slab} with all relevant types `bumped up by one'.
For instance, we could formulate the theorem for open-cover compact spaces assuming a suitable generalisation of $\SIND$.
Moreover, one could do the same for the other results in Section \ref{main} or even replace `$1\di 1$' by `$\sigma\di \sigma$', for any finite type $\sigma$.  
For instance, $\QFAC^{0, \NNN}$ is equivalent to the statement that continuous functions are bounded on sequentially compact $(M, d)$ with $M\subset \NNN$.   
We do not prove the latter result but obtain a nicer equivalence in Section \ref{boundz}.  We do point out that the existence of the supremum norm $\|f\|_{\infty}$ on the spaces from Theorem \ref{slabllll} is equivalent to Kleene's $(\exists^{4})$, which is $(\exists^{3})$ with all relevant types bumped up by one.  The same holds for all finite types \emph{mutatis mutandis}.  We stress that there may be more optimal\footnote{An anonymous referee has kindly suggested the following modification of the proof of Theorem~\ref{slabllll} that avoids $(\exists^{3})$ : instead of only mapping $o^{1\di 1}$ to $0$, one maps all $Y$ with $Y(0...)(0) = 0$ to $0$;  
instead of mapping $Y$ that are different from $o$ to $\frac{1}{2^{G(Y)}}$, one maps all $Y$ with $Y(0...)(0) \ne 0$ to $\frac{1}{2^{G(Z)}}$ for $Z(f)(n) := Y(f)(n+1)$.
} proofs of e.g.\ Theorem \ref{slabllll}.

\subsection{Axiom of countable Choice}\label{boundz}
We show that basic statements about sequentially compact metric spaces $(M, d)$ with $M\subset \NNN$ are equivalent to $\QFAC^{0, \NNN}$.  
This includes the intermediate value theorem as in Principle \ref{IVT}.  While equivalences for $\QFAC^{0,1}$ are not unheard of, the results in 
this section are entirely new.    

\smallskip

First of all, we use the historical definition of connectedness, due to Cantor and Jordan (\cite{wilders}), as in the second item of Definition \ref{chainc}.  
\bdefi[Connectedness]\label{chainc}~
\begin{itemize}
\item A metric space $(M, d)$ is \emph{connected} in case $M$ is not the disjoint union of two non-empty open sets.  
\item A metric space $(M, d)$ is \emph{chain connected} in case for any $w, v\in M$ and $\eps>0$, there is a sequence $w=x_{0},x_{1}, \dots, x_{n-1},{x_{n}}=v\in M$ such that for all $i< n$ we have $d(x_{i}, x_{i+1})<\eps$. 
\end{itemize}
\edefi
\noindent
It is well-known that the items in Def.\ \ref{chainc} are equivalent for compact metric spaces. 
\begin{princ}\label{IVT}
Let $(M, d)$ be a sequentially compact and chain connected metric space with $M\subset \NNN$ and let $f:M\di \R$ be continuous.
If $f(w)<c<f(v)$ for some $v, w\in M$ and $c\in \R$, then there is $u\in M$ with $f(u)=c$.  
\end{princ}
We observe that the base theory in the following theorem is much weaker than in Theorem \ref{slabllll}.
\begin{thm}[$\ACAo$]\label{NNNZ}
The following statements are equivalent.
\begin{enumerate}
\renewcommand{\theenumi}{\alph{enumi}}
\item The axiom of countable choice as in $\QFAC^{0,1\di 1}$.
\item The combination of the following.
\begin{itemize} 
\item[(b.1)] The axiom $\QFAC^{0,1\di 1}_{\fin}$: for quantifier-free $\varphi$ with $(\forall n^{0})(\exists Y^{1\di 1})\varphi(Y, n)$, we have $(\forall k^{0})(\exists (Y_{n})_{n\in \N})(\forall n\leq k)\varphi(Y_{n}, n)$.
\item[(b.2)] For any metric $d:\NNN^{2}\di \R$ and metric space $(M, d)$ with $M\subset \mathcal{N}$, sequential compactness implies boundedness. 
\end{itemize}
\item The combination of the following.
\begin{itemize} 
\item[(c.1)] The axiom $\QFAC^{0,1\di 1}_{\fin}$ from item \textup{(b.1)}.
\item[(c.2)] For any metric $d:\NNN^{2}\di \R$ and metric space $(M, d)$ with $M\subset\mathcal{N}$, sequential compactness implies countable compactness. 
\end{itemize}
\item The combination of the following.
\begin{itemize} 
\item[(d.1)] The axiom $\QFAC^{0,1\di 1}_{\fin}$ from item \textup{(b.1)}.
\item[(d.2)] The intermediate value theorem as in Principle \ref{IVT}.
\end{itemize}
\item The combination of the following.
\begin{itemize} 
\item[(e.1)] The axiom $\QFAC^{0,1\di 1}_{\fin}$ from item \textup{(b.1)}.
\item[(e.2)] For any metric $d:\NNN^{2}\di \R$ and sequentially compact metric space $(M, d)$ with $M\subset\mathcal{N}$, a \(Lipschitz\) continuous $F:M\di \R$ is bounded.
\end{itemize}
\end{enumerate}
\end{thm}
\begin{proof}
First of all, to obtain item (a) from the other items, let $\varphi$ be quantifier-free such that $(\forall n^{0})(\exists Y\in \NNN)\varphi(Y,n)$ but there is no sequence witnessing the latter. 
Now define a metric space $(M, d)$ with $M\subset \NNN$ as follows: let $Y^{1\di 1}\in M$ in case $Y(00\dots )(0)=n_{0}$ and $(\forall i\leq n_{0})(\varphi(\lambda f.P(Y, i)(f), i))$ where 
\[
P(Y, i)(f):=g_{i} \textup{ in case $Y(f)=\langle m\rangle * g_{0}\oplus \dots\oplus g_{n_{0}} $}.
\]  
Let $(q_{n})_{n\in \N}$ be an enumeration of the rationals (without repetitions).  
We put $Y=_{M}Z$ in case $Y(00\dots)(0)=Z(00\dots)(0)$ and define $d:\NNN^{2}\di \R$ as $d(Y, Z)=|q_{Y(00\dots)(0)}-q_{Z(00\dots)(0)}|$.  Then $(M, d)$ is readily seen to be an unbounded and not countably compact metric space, using $\QFAC^{0,1\di 1}_{\fin}$.  To show that $(M, d)$ is sequentially compact, let $(Y_{n})_{n\in \N}$ be a sequence in $M$.   
In case $(\forall n\in \N)(\exists m\in \N)(Y_{m}(00\dots)(0)>n)$, we obtain a sequence $(Z_{n})_{n\in \N}$ with $(\forall n\in \N)\varphi(Z_{n}, n)$, which contradicts our assumptions.  
Hence, we must have $(\exists n_{0}\in \N)(\forall m\in \N)(Y_{m}(0)\leq n_{0})$.  Now let $n_{1}\leq n_{0}$ be such that $Y_{m}(00\dots)(0)=n_{1}$ for infinitely many $m\in \N$.  
We thus obtain a sub-sequence of $(Y_{n})_{n\in \N}$ that is constant in $M$, and hence trivially convergent.    Thus, $(M, d)$ is a sequentially compact metric space that is unbounded and not countably compact, contradicting items (b) and (c); $\QFAC^{0,1\di 1}$ now follows from the latter.  

\smallskip

Regarding item (d), the chain connectedness of $(M, d)$ is proved using $\QFAC^{0, \NNN}_{\fin}$ as follows: fix $Y, Z\in M, \eps>0$ and consider $d(Y, Z)=|q_{Y(00\dots)(0)}-q_{Z(00\dots)(0)}|$.  
Let $q_{Y(00\dots)(0)}=r_{0},r_{1}, \dots, r_{k-1}, r_{k}=q_{Z(00\dots)(0)}\in \Q$ be a finite sequence such that $|r_{i}-r_{i+1}|<\eps$ for $i<k$.  Using $\QFAC^{0, \NNN}_{\fin}$, there are $Y_{i}\in M$ such that $q_{Y_{i}(00\dots)(0)}=r_{i}$ for $i<k$, and chain connectedness of $(M, d)$ follows. 
Now define $F:M\di \R$ by $F(Y)=\frac{1}{3}q_{Y(00\dots)(0)}$, which is Lipschitz continuous since for $Y, Z\in \NNN$:
\[\textstyle
|F(Y)-F(Z)|=|\frac{1}{3}q_{Y(00\dots)(0)}-\frac{1}{3}q_{Z(00\dots)(0)}|=\frac{1}{3}|q_{Y(00\dots)(0)}-q_{Z(00\dots)(0)}|\leq \frac{1}{3} d(Y, Z).
\]
However, the range of $F$ consists of rationals, i.e.\ it does not have the intermediate value property.  
This shows that item (d) implies $\QFAC^{0, \NNN}$.   Note that $F$ is also unbounded thanks to $\QFAC^{0, \NNN}_{\fin}$, i.e.\ item (e) also follows. 

\smallskip

For the remaining implactions, one uses the obvious proof-by-contradiction for items (b)-(c), and (e), namely as in the proofs of Theorems \ref{NNZ} and \ref{zirf}.  For item~(d), apply $\QFAC^{0, \NNN}$ to the definition of chain connectedness 
for $Y, Z\in M$ with $F(Y)<_{\R} z<_{\R} F(Z)$ for fixed $z\in \R$.  The resulting sequence readily yields a sequence with limit $W\in M$ such that $F(W)=_{\R}z$.
\end{proof}
With slight effort, one obtains the following corollary by `bumping down' all relevant types in the proof of the theorem.  
\begin{cor}[$\ACAo$] The following are equivalent. 
\begin{enumerate}
\item[(a)] The Axiom of countable Choice $\QFAC^{0,1}$.
\item[(b)] The combination of the following.
\begin{itemize} 
\item[(b.1)] The axiom $\QFAC^{0,1}_{\fin}$ from Theorem \ref{NNZ}.
\item[(b.2)] The intermediate value theorem as in Principle \ref{IVT} for $M\subset \N^{\N}$.
\end{itemize}
\end{enumerate}
\end{cor}
The previous equivalences are merely examples.  
One can similarly obtain equivalences as in Theorem \ref{NNNZ} involving the following items for sequentially compact metric spaces $(M, d)$, some of which are obvious by now.  
\begin{itemize}
\item Chain connectedness implies connectedness for $(M, d)$.
\item The previous item with `connectedness' replaced by `has finitely many connected components' or `not totally disconnected/separated'.
\item Locally constant $M\di \R$-functions are constant.
\item Continuous $M\di \{0,1\}$-functions are constant.
\item The approximate intermediate value theorem, i.e.\ with the consequent weakened to $(\forall \eps>0)(\exists Y\in M)(|F(Y)|<\eps)$.
\item If $M$ is chain connected and $|M|\geq 2$, then it is uncountable (\cite{trekker}*{p.\ 82}).
\item The principles $\ADS$ and $\CAC$ from the RM zoo (see \cite{dsliceke}) generalised to third-order orderings. 
\end{itemize}
For instance, one readily splits $M$ from the proof of Theorem \ref{NNNZ} into two open disjoints sets using an irrational number.  
As a more challenging result, a version of Nadler's fixed point theorem (\cite{naaldendraad}) should be equivalent to $\QFAC^{0,1}$.  

\smallskip

Finally, we show that Theorem \ref{NNNZ} remains true if we restrict to well-known special classes of metric spaces, as follows.
\bdefi\label{zek}
A metric space $(M, d)$ is called: 
\begin{itemize}
\item \emph{convex} if for any $x, y\in M$ with $x\ne_{M}y$, there is $z\in M$ such that $d(x,y)=d(x, z)+d(z,y)$ \(\cite{kirk1}*{p.\ 554}\), 
\item \emph{UMP} if for any $x, y\in M$ with $x\ne_{M}y$, there is a unique $z\in M$ such that $d(x,z)=d(z, y)$ \(\cite{naadloos}*{p.\ 353}\),  
\item \emph{\(Gromov\) hyperbolic} if there is $\delta>0$ such that for all $x, y, z, w\in M$, we have
\[
{ (x,z)_{w}\geq \min \left((x,y)_{w},(y,z)_{w}\right)-\delta },
\]
where ${ (y,z)_{x}={\frac {1}{2}}\left(d(x,y)+d(x,z)-d(y,z)\right)}$ is the \emph{Gromov product} \(\cite{ghs}*{p.\ 191}\).
\item \emph{Hadamard or CAT$(0)$} if $(M, d)$ is complete and for $x, y\in M$, there is $m\in M$ such that for all $z\in M$:
\be\label{konz}\textstyle
{ d(z,m)^{2}+{d(x,y)^{2} \over 4}\leq {d(z,x)^{2}+d(z,y)^{2} \over 2}.}
\ee
If equality holds in \eqref{konz}, the space $(M, d)$ is called \emph{flat Hadamard}.
\end{itemize}
\edefi
\begin{thm}\label{vinal}
Theorem \ref{NNNZ} holds restricted to convex, UMP, hyperbolic, or flat Hadamard spaces.  
\end{thm}
\begin{proof}
The theorem is established if we can show that the metric space $(M, d)$ from the first paragraph of the proof of Theorem \ref{NNNZ} satisfies the properties from Definition~\ref{zek}.
That this space is convex follows from the observation that for $Y, Z\in M$ with $Z\ne_{M}Y$, we have $q_{Y(00\dots)(0)}\ne q_{Z(00\dots)(0)}$. Now consider the midpoint $m=\frac{1}{2}( q_{Y(00\dots)(0)}+q_{z(00\dots)(0)} )$ and use $\QFAC^{0, 1\di 1}_{\fin}$ to find $W\in M$ with $q_{W(00\dots)(0)}=m$, i.e.\ we have 
\[
d(Z, Y)=|q_{Y(00\dots)(0)}-q_{Z(00\dots)(0)}|= |q_{Y(00\dots)(0)} -m| + |m- q_{Z(00\dots)(0)}| = d(Y, W)+d(W, Z)
\]
as required for convexity.  The midpoint $W$ is clearly unique and trivially $d(Y, W)=d(W, Z)$, i.e.\ we also obtain UMP.  
Moreover, one readily verifies that \eqref{konz} holds with equality for the midpoint $W$.  
To show that $(M, d)$ is hyperbolic, one proceeds via a tedious-but-straightforward case distinction for $\delta=1$.  
To show that $(M, d)$ is complete, one verifies that all Cauchy sequences are eventually constant in the absence of the choice function as in $\QFAC^{0, 1\di 1}$.  
\end{proof}

\smallskip

In conclusion, we have obtained equivalences for $\QFAC^{0, 1\di 1}$ involving the intermediate value theorem in Theorem \ref{NNNZ}.  
It is fairly straightforward to generalise these results for `$1\di 1$' replaced by `$\sigma\di \sigma$', for any finite type $\sigma$.

\begin{ack}\rm 
We thank the anonymous referee for providing the many helpful suggesting that improved this paper.  
Our research was supported by the \emph{Klaus Tschira Boost Fund} via the grant Projekt KT43.
The initial ideas for this paper were developed in my 2022 Habilitation thesis at TU Darmstadt (\cite{samhabil}) under the guidance of Ulrich Kohlenbach.  
The main ideas of this paper came to the fore during the \emph{Trends in Proof Theory} workshop in February 2024 at TU Vienna.  
Theorem \ref{vinal} was obtained during the \emph{Workshop on Proof mining} at TU Darmstadt in Sept.\ 2024. 
We express our gratitude towards all above persons and institutions.   
\end{ack}

%

\bye